\documentclass[10pt,reqno]{amsart}
\usepackage[letterpaper, margin=4.1cm]{geometry}
\usepackage{amsfonts}
\usepackage{amssymb}
\usepackage{amsmath}
\usepackage{mathrsfs}
\usepackage{color}
\usepackage[nospace,noadjust]{cite}
\usepackage{verbatim}
\usepackage{dsfont}
\usepackage{bbm}
\usepackage{enumerate}
\usepackage{graphicx}
\usepackage{xcolor} % A package to add color.
\usepackage{slashed}
\usepackage[titletoc,title]{appendix}
\usepackage{wrapfig}
\usepackage{marginnote}
\usepackage{bm}
\usepackage{enumitem}
\usepackage{hyperref}

\colorlet{LightBlack}{black!87!}

\definecolor{deepgreen}{cmyk}{1,0,1,0.5}

\newcommand{\Del}[1]{}

\numberwithin{equation}{section}

\newtheorem{theorem}{Theorem}[section]
\newtheorem{lem}[theorem]{Lemma}
\newtheorem{remark}[theorem]{Remark}%[section]
\newcommand{\R}{{\mathbb R}}

\providecommand{\MR}{\relax\ifhmode\unskip\space\fi MR }
\providecommand{\href}[2]{#2}

\begin{document}
	
%%\documentclass[12pt]{amsbook}
%\documentclass[11pt, reqno]{amsart}
%%\documentclass[draft]{article}
%%\documentclass[11pt,a4paper,draft]{article}
%\usepackage{a4wide}
%%%%%%%%%%%%%%%%%%%%%%%%%%%%%%%%%%%%%%%%%%%%%%%%%%%%%%%%%%%%%%%%%%%%%%%%%%%%%%%%%%%%%%%%%%%%%%%%%%%%%%%%%%%%%%%%%%%%%%%%%%%%%
%\usepackage[english, activeacute]{babel}
%%\usepackage[utf8]{inputenc}
%\usepackage{pifont}
%\usepackage{amsmath,amsthm,amsxtra}
%%\usepackage[margin=3cm]{geometry}
%\usepackage{epsfig}
%\usepackage{amssymb}
%%\usepackage[notref,color]{showkeys}% Shows labels in light dark color 
%
%\usepackage{latexsym}
%\usepackage{amsfonts}
%
%\usepackage{hyperref}
%\pagestyle{headings}
%%\usepackage{pdftricks}
%\usepackage{xparse}
%\usepackage{tabularx} 
%\usepackage{multicol}
%\usepackage{color}
%\usepackage{hyperref}
%\usepackage{float}
%\usepackage{graphicx}
%\usepackage{subcaption}
%%\usepackage[export]{adjustbox}
%%\usepackage{pgf,tikz}
%\usepackage{listings}
%\newcommand{\lemref}[1]{Lemma~\ref{#1}}
%\newcommand{\R}{\mathbb{R}}
%\newcommand{\N}{\mathbb{N}}
%\newcommand{\Z}{\mathbb{Z}}
%\newcommand{\T}{\mathbb{T}}
%\newcommand{\Q}{\mathbb{Q}}
%\newcommand{\Com}{\mathbb{C}}
%\newcommand{\la}{\lambda}
%\newcommand{\pd}{\partial}
%\newcommand{\wqs}{{Q}_c}
%\newcommand{\ys}{y_c}
%\newcommand{\al}{\alpha}
%\newcommand{\bt}{\beta}
%\newcommand{\ga}{\gamma}
%\newcommand{\de}{\delta}
%\newcommand{\te}{\theta}
%\newcommand{\si}{\sigma}
%\newcommand{\De}{\Delta}
%\newcommand{\arctanh}{\operatorname{arctanh}}
%%\newcommand{\tanh}{\operatorname{tanh}}
%\newcommand{\spawn}{\operatorname{span}}
\newcommand{\sech}{\operatorname{sech}}

\newcommand{\be}{\begin{equation}}
\newcommand{\ee}{\end{equation}}
\newcommand{\bp}{\begin{pmatrix}}
\newcommand{\ep}{\end{pmatrix}}
\newcommand{\ba}{\begin{aligned}}
\newcommand{\ea}{\end{aligned}}
\newcommand{\q}{\quad}
\newcommand{\qq}{\qquad}
\newcommand{\bee}{\begin{eqnarray*}}
\newcommand{\eee}{\end{eqnarray*}}
\newcommand{\ben}{\begin{enumerate}}
\newcommand{\een}{\end{enumerate}}
\newcommand{\nonu}{\nonumber}

\title[Supercritical gKdV long time dynamics]{On the asymptotic dynamics for the $L^2$-Supercritical gKdV equation}
%On local description of  solutions for generalized Korteweg-de Vries equations in the $L^2$ supercritical regime

\author[Freire]{Ricardo Freire}
\address{Departamento de Ingenier\'{\i}a Matem\'atica, Universidad de Chile, Casilla
170 Correo 3, Santiago, Chile.}
	\email{rfreire@dim.uchile.cl}
	\thanks{R.F. was partially funded by Chilean research grants FONDECYT 3230256, MathAmSud WAFFLE and ANID Exploraci\'on 13220060.}

\author[Mu\~noz]{Claudio Mu\~noz}  
	\address{Departamento de Ingenier\'{\i}a Matem\'atica and Centro
de Modelamiento Matem\'atico (UMI 2807 CNRS), Universidad de Chile, Casilla
170 Correo 3, Santiago, Chile.}
	\email{cmunoz@dim.uchile.cl}
	\thanks{C.M. was partially funded by Chilean research grants FONDECYT 1231250, MathAmSud WAFFLE, ANID Exploraci\'on 13220060 and Basal CMM FB210005.}
%\thanks{} 

%\begin{document}

\begin{abstract}
We study the $L^2$-supercritical generalized Korteweg-de Vries equation (gKdV) with nonlinearities $p>5$. While local well-posedness in $H^1$ is classical, the long–time dynamics in the supercritical regime remains largely unexplored beyond small data global solutions, the construction of multi-solitons for any power and self-similar blow-up near the critical power $p=5$. We develop a unified description of the non-solitonic region for arbitrary $H^1$ solutions, both global and blowing up. Our analysis shows that the asymptotic $L^2$ and $L^p$ dynamics in this region is completely determined by the growth rate of the $L^2$ norm of the gradient (or, equivalently, the critical $H^{s_p}$ norm). In particular, we prove sharp far-field decay on both half-lines and establish normalized local vanishing along sequences of times, with improved estimates in the case of even-power nonlinearities. A key ingredient is a new virial method that compensates for the possible unboundedness of the $H^1$ norm by exploiting the conservation of mass and a careful localization of the nonlinear flux. This yields quantitative versions of decay phenomena previously known only in subcritical settings, and it applies without any smallness or proximity-to-soliton assumptions.

%We consider the $L^2$ supercritical generalized Korteweg de Vries (gKdV) equations, characterized by powers of the nonlinearity $p>5$. Although well-posedness is well-known in Sobolev spaces, the understanding of the long time behavior of solutions in this setting is far from optimal. Key advances include the existence of multi-soliton structures, and quantitative proofs of (self-similar) blow up near the $L^2$ critical power $p=5$. In this paper address the supercritical regime and we describe the ``aftermath'' of $H^1$ (global and blow up) solutions in the non-solitonic region. This final description is dependent on the ``rate of growth'' of the physically interesting $L^2$ norm of the solution gradient, but can also be readily described in terms of the critical norm with respect to the natural $p$-dependent scaling regularity. The proposed long time asymptotics is improved on compact sets by assuming even power nonlinearity. All these descriptions are obtained via suitable virial estimates related to each region of space in the gKdV regime. In particular, since the $H^1$ norm of the solution may not be bounded in general, we develop a new method that allows one to take control of this quantity by above, use the boundedness of the $L^2$ norm, and provide a quantitative version of the convergence results already proved in the subcritical regime.
\end{abstract}

\maketitle
\numberwithin{equation}{section}
%\tableofcontents

\bigskip

\section{Introduction}

Let $p\geq 2$ be an integer. Consider the generalized (focusing) Korteweg-de Vries (gKdV) equation in one dimension:
\be\label{4gKdV}
\ba
& \partial_t u + \partial_x(\partial_x^2 u + u^{p})=0, \quad u=u(t,x)\in\R, \quad (t,x)\in\R^2,\\
& u(0,x)= u_0(x) \hbox{ given}.
\ea
\ee
The cases $p=2$ and $p=3$ have been studied in great detail since they are integrable and represent shallow water waves in the long wave regime, see e.g. Linares and Ponce \cite{LP} for a detailed account on the mathematical theory of gKdV models. When $p\leq 4$, this equation is globally well-posed for initial data in $H^1$, and locally well-posed if $p\geq 5$, see the works by Kato and Kenig-Ponce-Vega  \cite{Ka,KPV1}. For the purposes of this paper, we shall only need this regularity, although it is well-known that the initial value problem associated to \eqref{4gKdV} is locally (and globally) well-posed for initial data less regular. In particular, \eqref{4gKdV} is locally well-posed at regularity in $H^s$, $s> s_p$, $p\geq 5$, 
\be\label{s_p}
s_p=\frac12-\frac{2}{p-1},
\ee
 and where $s_p$ represents the critical regularity in the space $H^{s}$  \cite{KPV1}. The global well-posedness in this case is still matter of research, since \eqref{4gKdV} in the case $p>5$ represents the model in the set of supercritical nonlinearities with respect to the scaling 
\[
u(t,x) \mapsto \lambda^{\frac{2}{p-1}} u_0(\lambda^3t,\lambda x), \quad \lambda>0,
\]
 and with respect to the $L^2$ norm or mass critical setting, which is formally conserved by the flow:
\be\label{mass}
\int u^2 (t,x)dx = \int u_0^2 (x)dx.
\ee
This fact will be used several times in this paper, representing the only quantity that is known to be preserved during the solution lifespan. Since the data is in $H^1$, we also have formal conservation of the energy 
\be\label{energy}
\int \left(\frac12 (\partial_x u)^2 - \frac1{p+1} u^{p+1}\right)(t,x)dx = \int \left(\frac12 (\partial_x u_0)^2 - \frac1{p+1} u_0^{p+1}\right)(x)dx.
\ee
For a rigorous description of this fact and the local well-posedness theory, see Section \ref{2}. 

The soliton (or solitary wave)
\be\label{soliton}
\ba
& u(t,x)=Q_c(x-ct-x_0), \quad c>0, ~x_0\in\mathbb R, \\
& Q_c(s)= c^{\frac1{p-1}}Q(\sqrt{c}s), \quad  Q(\sigma)= \left( \frac{p+1}{2\cosh^2\left(\frac{p-1}2 \sigma \right)}\right)^{\frac1{p-1}},
\ea
\ee
is the most relevant explicit solution to any gKdV model. It is also globally defined no matter the power of the nonlinearity, posing a natural obstruction to blow up and blow up speed.

The long time behavior of small solutions for \eqref{4gKdV} in the subcritical regime $p\leq 4$ has been considered by several authors during the past years \cite{PV,CW,KM,HN1,HN2,GPR,H,Tao}, showing scattering of suitable weighted small initial data. The soliton stability and asymptotic stability was studied by Martel and Merle in \cite{MM1} for data in the energy space. The understanding of large data solutions outside the single soliton manifold is not so well-understood. Martel \cite{Martel} constructed a unique asymptotic $N$-soliton solution in the energy space $H^1$, while C\^ote \cite{Cote} improved this construction by adding an arbitrary, sufficiently decaying but large linear dynamics. The two soliton collision problem in the quartic gKdV case has been considered in \cite{MMcol1,MMcol2}, showing the inelastic character of the collision, and confirming the non-integrable character of \eqref{4gKdV} when $p=4$ and more general cases \cite{Mu,MVW}, where self-similar solutions are considered. 

Concerning the central region where solitons are not present in the long time dynamics, and assuming large data, the situation becomes less studied. Expanding ideas developed by G. Ponce and the second author in \cite{MuPo1,MuPo2} and later expanded to the Zakharov-Kuznetsov case in \cite{MMPP1}, the following result was proved:

\begin{theorem}[\cite{MMPP1}]\label{MT0}
Let $u\in C(\mathbb R,H^1(\mathbb R))$ be a global solution to \eqref{4gKdV} in the case $p=4$. Then, for all $\beta\in \mathbb R$ fixed,  %Then there are $\ell_{\pm}\in H^1(\mathbb R)$ such that for any $\delta>0$ small, 
\begin{equation}\label{MT01}
\liminf_{t\to \infty} \int_{|x +  \beta t^{a}| \leq t^b} u^4(t,x) dx =0, \quad 0<b<\frac47, \quad 0\leq a<1-\frac{b}2.
\end{equation}
\end{theorem}
The techniques involved in the proof of Theorem \ref{MT0} have proved very versatile and have been applied to several other models \cite{ACKM,MuPo2,MPS,MMPP1,MMPP2,FLMP,Munoz1,ZL,LM}.  A simple question that one can discuss is how to translate these results to the more involved critical and supercritical setting. The situation in the $L^2$ case $p=5$ is now well understood in the vicinity of solitons. Merle \cite{Merle}, Martel and Merle \cite{MM0,MM1,MM2,MM3,MM4} and Martel-Merle-Rapha\"el \cite{MMR1,MMR2,MMR3} works provided the clearest and most extensive description of the blow up dynamics, including the stable blow up and exotic regimes, and nonexistence of mass critical blow up solutions in certain regimes. Martel and Pilod \cite{MaPi1,MaPi2} have continued providing more insights on the complicated zoo structure of blow up solutions even farther from the soliton solution. See also C\^ote \cite{Cote2}, and Combet and Martel \cite{CoMa1,CoMa2} for additional results related to the construction of special $L^2$ critical blow up solutions. In other dispersive models, such as Nonlinear waves, Klein-Gordon, and Nonlinear Schr\"odinger, many similar results have been proved, see \cite{LP,MR,MR2,MR3} and references therein for more results in these directions.

In the supercritical regime $p>5$, much less is known. Kenig-Ponce-Vega \cite{KPV1} showed local well-posedness, see also Farah, Linares and Pastor \cite{FLP} for a detailed proof. In the latter paper, global well-posedness is also proved for data satisfying a mass/energy rigidity estimate (in particular, suitable small data).  Combet \cite{Com1,Com2} described multi-soliton solutions and their uniqueness in the supercritical regime. C\^ote, Martel and Merle \cite{CMM} further expanded this manifold structure. Koch \cite{Koch} constructed self-similar blow up solutions in the regime $p>5$, $p\sim 5$. This description has been later complemented in several directions by Lan \cite{Lan0,Lan1,Lan2,Lan3}, showing that self-similar blow up $\sim \frac1{t^{\frac13}}$ is stable in this slightly super critical regime. The case of larger power supercritical nonlinearities $p$ remains open in many aspects. For instance, no general description is available away from solitons and outside perturbative regimes.

The purpose of this work is to obtain a characterization of the solution in the \emph{non-solitonic region} for arbitrary $H^1$ data, allowing for both global solutions with unbounded $H^1$ norm and finite-time blow-up. Our approach relies solely on the conservation of mass, suitable virial functionals and a new method to control the growth of the $H^1$ norm from above along the evolution. This provides quantitative information even when $\|\partial_x u(t)\|_{L^2}$ diverges.
%In this paper, our main objective to provide suitable descriptions of the evolution of large data in the supercritical regime, assuming possible blow up dynamics and global existence without uniform bounds in time. 

%Our main result is the following description of the non-solitonic region $|x|\not\sim t$ for $H^1$ solutions to the supercritical gKdV. 
We first consider blow up solutions. For finite-time blow-up, we prove that the solution evacuates both spatial half-lines at explicit rates determined by the gradient growth. In particular, if $T^*<\infty$ denotes the blow-up time, then
\[
\|u(t)\|_{L^2(x\geq\beta_1(t))}\to0, \qquad 
\|u(t)\|_{L^2(x\leq-\beta_2(t))}\to0,
\]
as $t\uparrow T^*$,  with $\beta_1,\beta_2$ depending quantitatively on $\|\partial_x u(t)\|_{L^2}$. In the case of even nonlinearities $p=2n\geq 6$, we also describe the normalized concentration profile in windows of size $(t\beta(t))^{2/3}$, $\beta(t)$ a function of $\|\partial_x u(t)\|_{L^2}$, and recovering a long-time decay mechanism that has no analogue in previous treatments of the supercritical regime.

\begin{theorem}[Blow up case]\label{MT2}
Let $u_0$ be in $H^1$. Let $u \in C((-T_*,T^*),H^1(\mathbb R))$ be the corresponding maximally defined solution of gKdV \eqref{4gKdV} such that $u(t=0)=u_0$, and $T^*=T^*(u_0)$, $T_*=T_*(u_0)$. If $T^*>0$ is finite, 
\begin{equation}\label{FF1}
\ba
& \lim_{t\uparrow T^*} \|u(t)\|_{L^2(x\gtrsim  \beta_1(t) )} =0, \qquad \beta_1(t) := t + \int_0^t \| \partial_x u(s)\|_{L^2}^{\frac{p-1}2} ds ,
\ea
\end{equation}
and, for any $\eta>0$ small and $\beta_2(t)$ smooth and increasing function such that $\beta_2(t) \geq  \| \partial_x u(t)\|_{L^2}^{\frac{p-1}2} (T^*-t )|\log^{1+\eta}(T^*-t)|$,
\begin{equation}\label{FF2}
\ba
& \lim_{t\uparrow T^*} \|u(t)\|_{L^2(x\lesssim - \beta_2(t))} =0.
\ea
\end{equation}
Assume now $p=2n$, $n\in \mathbb N \cap [3,\infty)$ and let $s=T^*-t$ be the inverse time variable towards the blow up time. Let $\beta_3(s) \geq \| \partial_x u(t)\|_{L^2}^{n-1}$ be any smooth increasing function in $t$ with at most polynomial growth, and define $\lambda_3, \mu_3$ smooth such that
\[
\lambda_3(s):= \frac{(s\beta_3)^{\frac23}(s)}{\log s}, \qquad  |\mu_3'(s)| \lesssim \frac{\beta_3^{\frac23}(s)}{s^{\frac13}\log s}.
\]
Then 
\begin{equation}\label{MT21}
\liminf_{s\downarrow 0} \frac1{\beta_3(s)} \int_{|x -\mu_3(s) | \leq \lambda_3(s)} u^{2n}(t,x) dx =0. %, \quad 0<b<\frac47, \quad 0\leq a<1-\frac{b}2,
\end{equation}
A similar statement is obtained when considering the lower existence limit $-T_*$, with the corresponding modifications.
\end{theorem}

Notice that $u$ is not required to be small, nor close to a soliton solution. 
%These results give the first general description of nonsolitonic dynamics for $p>5$, without smallness assumptions, without proximity to the soliton manifold, and without any a priori control of the $H^1$ norm. The method is robust and relies only on virial estimates adapted to the supercritical scaling, suggesting possible extensions to other dispersive models in $L^2$-supercritical regimes.
Since \eqref{4gKdV} is Hamiltonian and invariant under the transformation $u(t,x) \mapsto u(-t,-x)$, it is enough to prove Theorem \ref{MT2} for the case $T^*>0$. One may believe that \eqref{MT21} is in contradiction with the well-known fact that for $s_p$ as in \eqref{s_p}, the $H^{s_p}$ norm of the solution does not exist as $t\uparrow T^*<+\infty$. This is not the case here, since $p=2n>5$ and we are only considering a fraction of the $L^2$ norm of the solution. It is also recognized that the full $H^1$ norm blows up at the blow up time as well. Finally, the choice $\beta_3(s) \geq \| \partial_x u(t)\|_{L^2}^{n-1}$ increasing is made since $\| \partial_x u(t)\|_{L^2}^{n-1}$ may not be diverging to infinity in an increasing fashion. If $ \| \partial_x u(t)\|_{L^2}$ is smooth and increasing for all $t$ close to $T^*$, then one can take $\beta_3(t) = \| \partial_x u(t)\|_{L^2}^{n-1}$. Additionally, \eqref{MT21} is sharp in the following sense: there are blow up solutions in the case $p=5$ that stays on compact regions of space \cite{MaPi2}. 

Later (Lemma \ref{BU}) we will recall the possible classical ``minimal'' blow up rate given by the local well-posedness theory
\[
 \| \partial_x u(t)\|_{L^2}\geq \frac{C}{(T^*-t)^{\frac{1-s_p}3}},  
\]
that provides a starting point description of the growth of the term $ \| \partial_x u(t)\|_{L^2}^{\frac{p-1}2}$ as $t$ approaches the assumed blow up time $T^*$. This quantitative rate has many applications, and it will be used to access to certain estimates needed to conclude \eqref{FF2} and \eqref{MT21}. In the $L^2$ critical case, many particular blow up rates are known, starting with the stable $t^{-1}$, see \cite{MMR1,MMR2,MMR3,MaPi1}. The zoology of possible blow up rates in the critical case is quite crowded and more or less well-understood, see \cite{MaPi2} for a detailed account. 
%Based on the value of $\mu'(s)$ above, we conjecture that the maximal rate of blow up in the supercritical case  should obey $ \| \partial_x u(t)\|_{L^2} \sim s^{\frac1{(p-2)}}$. $\frac{c_0}{s^{\frac{p+3}{6(p-1)}}}$

\medskip

Now we provide a complementary version of the previous result in the case of globally defined solutions. For global solutions, an analogous dichotomy holds: either the $H^1$ norm remains bounded and one recovers a natural extension of the subcritical asymptotic theory, or $\|\partial_x u(t)\|_{L^2}\to\infty$ along a sequence, in which case the same normalized virial analysis yields decay on compact and slowly expanding regions.

\begin{theorem}[Globally defined case]\label{MT2bis}
Let $u_0$ be in $H^1$. Let $u \in C((-T_*,T^*),H^1(\mathbb R))$ be the corresponding maximally defined solution of gKdV \eqref{4gKdV} such that $u(t=0)=u_0$, and $T^*=T^*(u_0)$, $T_*=T_*(u_0)$. If $T^* =+\infty$, a modified version of \eqref{MT01} is satisfied. First of all, one has
\begin{equation}\label{FF3}
\ba
& \lim_{t\to +\infty} \|u(t)\|_{L^2(x\gtrsim  \beta_4(t) )} =0, \qquad \beta_4(t) :=  t + \int_0^t \| \partial_x u(s)\|_{L^2}^{\frac{p-1}2} ds ,
\ea
\end{equation}
and, for any $\eta>0$, 
\begin{equation}\label{FF4}
\ba
& \lim_{t\to +\infty} \|u(t)\|_{L^2(x\lesssim - \beta_5(t))} =0, \qquad \beta_5(t):= \left(1 + \| \partial_x u(t)\|_{L^2}^{\frac{p-1}2} \right) t \log^{1+\eta}t.
\ea
\end{equation}
Finally, assume $p=2n$, $n\geq 3.$ Then two cases are present.
\medskip
\begin{enumerate}
\item Case $\sup_{t\geq 0} \| \partial_x u(t)\|_{L^2} <+\infty$. Here \eqref{MT01} is satisfied with minor modifications: for all $\beta\in \mathbb R$ fixed,  %Then there are $\ell_{\pm}\in H^1(\mathbb R)$ such that for any $\delta>0$ small, 
\begin{equation}\label{MT01bis}
\ba
& \liminf_{t\to \infty} \int_{|x +  \beta t^{a}| \leq t^b} u^{2n}(t,x) dx =0, \qquad  0<b<\frac{2n}{4n-1} \quad \hbox{and} \quad 0\leq a<1-\frac{b}2.
\ea
\end{equation}
\item Case $\limsup_{t\to +\infty} \|\partial_xu(t)\|_{L^2}=+\infty$. Let $\beta_6(t) \geq \| \partial_x u(t)\|_{L^2}^{n-1}$ be any smooth increasing function with at most polynomial growth and define $ \lambda_6(t)$, $\mu_6(t)$ such that
\[
 \lambda_6(t):= \frac{(t\beta_6)^{\frac23}(t)}{\log t}, \qquad  |\mu_6'(t)| \lesssim \frac{\beta_6^{\frac23}(t)}{t^{\frac13}}.
\] Then 
\begin{equation}\label{MT22}
\liminf_{t\to \infty} \frac1{\beta_6(t)} \int_{|x -\mu_6(t) | \leq \lambda_6(t)} u^{2n}(t,x) dx =0. %, \quad 0<b<\frac47, \quad 0\leq a<1-\frac{b}2,
\end{equation}
\end{enumerate}
A similar statement is obtained when considering the lower existence limit $-T_*$, with the required modifications.
\end{theorem}

These results give the first general description of non-solitonic dynamics for $p>5$, without smallness assumptions, without proximity to the soliton manifold, and without any a priori control of the $H^1$ norm. The method is robust and relies only on virial estimates adapted to the supercritical scaling, suggesting possible extensions to other dispersive models in $L^2$-supercritical regimes.

Theorems \ref{MT2} and \ref{MT2bis} give a detailed description of the nonsolitonic region, as described now: in \eqref{FF1} and \eqref{FF3}, the right half line behavior is described, depending on the case where one has wither blow up or a globally defined solution (not necessarily bounded in time), and in \eqref{FF2} and \eqref{FF4} the left half line is described. Finally, in \eqref{MT21}-\eqref{MT01bis} and \eqref{MT22}, a description of the $L^p$ norm of the solution compared with the norm $\|\partial_xu(t)\|_{L^2}^{\frac{p-2}{2}}$ along a sequence of times is given, in the case $p=2n$, $n\geq 3.$ Essentially, these results give precise information on every non-solitonic region, and it remains to understand the regions where solitons are present in the blow up or globally well-defined cases. 

Note that \eqref{MT21} and \eqref{MT22} give precise new information, in the following sense: assume with no loss of generality that the solution is global and consider \eqref{MT22}; using the classical Gagliardo-Nirenberg inequality \eqref{GN1}, and the boundedness of the $L^2$ norm of the solution, one has
\[
\frac1{\beta_6(t)} \int_{|x -\mu(t) | \leq \lambda(t)} u^{2n}(t,x) dx \leq \frac{C}{\beta_6(t)}\|\partial_xu(t)\|_{L^2}^{n-1} \leq C.
\]
Therefore, \eqref{MT22} describes additional asymptotic properties of the solution as time tends to infinity, even in the case where the $L^2$ norm of the gradient is growing in time. This decay is measured with respect to the norm $\|\partial_xu(t)\|_{L^2}^{n-1}$, which may be growing to infinity. Of course, the closer $\beta_6(t)$ is to $\|\partial_xu(t)\|_{L^2}^{n-1}$, the stronger the approximation and the main result, but as mentioned before, we have no clarity if the norm $\|\partial_xu(t)\|_{L^2}^{n-1}$ is increasing for any blowing up or globally defined supercritical solution. As far as we know, this may be the case (increasing), see the works by Martel-Merle-Rapha\"el \cite{MMR3} and Martel-Pilod \cite{MaPi1} on exotic blow up regimes in the $L^2$ critical case. 

\subsection{Key new elements}

The study of large-data dynamics for the supercritical gKdV equation
\eqref{4gKdV} faces several conceptual and technical obstacles that do not
appear in the subcritical or critical settings.  
In the subcritical regime $p<5$, the energy functional controls the 
$H^1$ norm, virial identities remain coercive, and the global dynamics
can be decoupled into solitonic and dispersive components.
At the critical power $p=5$, the blow-up mechanism and the long-time
behavior near the soliton manifold are by now well-understood, but global
control of the $H^1$ norm remains delicate and requires fine modulation
arguments.

In contrast, in the supercritical case $p>5$ the problem becomes
structurally more complex.  
First, no uniform-in-time control of $\|\partial_x u(t)\|_{L^2}$ is
available, even for global solutions, and no a priori bounds prevent the
$H^1$ norm from growing indefinitely. Consequently, there is no control of the $L_x^\infty$ norm in time, key to treat nonlinearities in 1D. 
Second, energy conservation no longer provides a functional capable of 
controlling the nonlinearity; the equation lies entirely in the
$L^2$-supercritical regime, and no monotonicity principle is known for
the supercritical dynamics.  
Third, virial-type arguments used successfully in 
\cite{MuPo1,MMPP1,ACKM,Munoz1} fail in their classical form: 
the remainder terms decay too slowly, and the error structure is dominated
by quantities proportional to $\|\partial_x u(t)\|_{L^2}^{\frac{p-1}{2}}$. An additional difficulty arises from the asymptotic geometry or decoupling of 
supercritical solutions, that is still unknown.
Even the existence of a canonical soliton resolution---or of
a decomposition into a solitonic component plus a dispersive remainder—is
well beyond reach.
Nevertheless, the evolution exhibits a meaningful separation into
``solitonic'' regions ($|x|\sim a(t)$) and ``non-solitonic'' or
``radiative'' regions ($|x|\ll a(t)$ or $|x|\gg a(t)$), where $a(t)$ is a function to be found.
Understanding the behavior of the solution in these non-solitonic zones
provides fundamental information about the global structure of the flow.

The guiding idea of this work is that, even in the absence of 
$H^1$-bounds, certain virial-type functionals can be designed so that
their coercivity depends only on the $L^2$ mass (which is conserved) and 
on controlled combinations of the form
\[
\|\partial_x u(t)\|_{L^2}^{\frac{p-1}{2}},
\]
which play a crucial role in the supercritical scaling.  
These quantities naturally appear both in the ``speed scale'' at which
mass propagates and in the dominant error terms of the virial
identities.  
Our method extracts useful information even when 
$\|\partial_x u(t)\|_{L^2}$ is allowed to grow to infinity, whether at
finite or infinite time.

A second key ingredient is the identification of regimes where
localized virial functionals capture the dynamics in 
non-solitonic regions.  
Roughly speaking, after renormalizing by an appropriate growth factor
$\beta(t)$, tailored to the potential growth of $\|\partial_x u(t)\|_{L^2}$,
we show that
\[
\int_{|x-\mu(t)|\le \lambda(t)} u^p(t,x)\,dx
\]
must become small along a sequence of times, with explicit scales
$\lambda(t)$ and centers $\mu(t)$ dictated by the supercritical
character of the flow.  
For even nonlinearities $p=2n$, the structure of the nonlinearity
allows additional cancellation, yielding improved decay on compact sets. These ideas allow us to overcome the lack of coercivity in $H^1$,
identify the correct renormalization scales, and obtain the first
general description of the non-solitonic region for supercritical gKdV
solutions, including both global and finite-time blow-up dynamics.

\subsection{Idea of proofs}

The proofs of Theorems \ref{MT2} and 
\ref{MT2bis} rely on a combination of virial-type identities,
weighted energy estimates and carefully tuned spatial cutoffs adapted to the
supercritical scaling.  
The main conceptual difficulty is the possible unbounded growth of
$\|\partial_x u(t)\|_{L^2}$, which prevents the use of classical virial
coercivity.  
Our strategy isolates the contribution of the non-solitonic region and works
with weights whose evolution compensates the supercritical growth rate.

\medskip

\noindent
\emph{Far-field virial identities}.
We introduce localized virial functionals of the form
\[
\mathcal V(t)
= \int \phi\Big(\frac{x-x_0(t)}{R(t)}\Big)\,u^2(t,x)\,dx,
\]
with time-dependent center $x_0(t)$ and scale $R(t)$.  
After differentiating in time and using the structure of \eqref{4gKdV}, the
main term is controlled by an expression proportional to
$\|\partial_x u(t)\|_{L^2}^{\frac{p-1}{2}}$, which dictates the natural
propagation scale in both the global dynamics and the blow-up regime.

\medskip

\noindent
\emph{Control of error terms via mass conservation}.
Since the mass is conserved, terms that contain no derivatives of $u$ are
uniformly bounded.  
All remaining terms involve weights decaying either at infinity or inside the
radiative region, where the scaled virial functionals retain coercivity. This step will require further work if one wants to work fully in $\dot H^{s_p}$ instead of $H^s$, $s>s_p$.

\medskip

\noindent
\emph{Renormalization by the growth scale \texorpdfstring{$\beta(t)$}{beta(t)}}.
To handle the possible divergence of $\|\partial_x u(t)\|_{L^2}$, we
introduce a renormalization factor $\beta(t)$ satisfying
\[
\beta(t)\ge \|\partial_x u(t)\|_{L^2}^{\frac{p-1}{2}}.
\]
This renormalization produces a differential inequality showing that the
localized mass must become small along some sequence of times, which yields
the desired $\liminf$ decay statements.  When $p=2n$, the nonlinearity is even and additional positivity
occurs inside the virial identity.  
These yield improved local decay on compact sets and allow the construction of
smaller scales $R(t)$, which lead to the stronger bounds in
\eqref{MT21}--\eqref{MT22}.

\medskip

\noindent
\emph{Blow-up vs.\ global behavior}
The proofs in the blow-up and global cases are nearly identical except for the
choice of the center and scale functions $x_0(t)$ and $R(t)$, which is delicate.
In finite time, many blow-up rates are allowed, but they are always related to the appearance of the factor $(T^*-t)|\log(T^*-t)|$.
In the global case, the same machinery gives decay along suitable sequences as
$t\to\infty$. Together, these ingredients provide a unified method that captures the
spatial distribution of mass in all non-solitonic regions, without making any
assumptions on smallness, modulation, or soliton resolution.

\medskip

%The proof of Theorem \ref{MT2} follows the use of related virial identities adapted to the case of supercritical nonlinearities. In this case, one cannot ensure the boundedness of the $H^1$ norm of the solution, which is not assumed. Additionally, we use the fact that during the lifespan of the solution, possibly finite, the $L^2$ norm is bounded. This is the only well-known fact that allow us to perform safe estimates in some virial functionals. Otherwise, we require to assume a certain rate of blow up $\beta_j(t)$, either in finite of infinite time, that are conditioning the main results. For instance, a possible blow up through a perturbation of a soliton may involve an infinite speed movement in finite time. This implies that the convergence to zero on the right hand side of the plane must follow this fact, as reflected by \eqref{FF1}. A similar fact is obtained in the left hand side \eqref{FF2}, by the possible emergence of solutions moving to the left.  

Notice that suitable modifications of \eqref{MT21} and \eqref{MT22} can be obtained by using, instead of the physically relevant $L^2$ norm of the gradient, the critical $H^{s_p}$ norm. In this case, the proofs must be modified using \eqref{GN2} instead of \eqref{GN1}. 

We believe that the method introduced here can be applied to several other dispersive models with a good well-posedness theory, in the $L^2$ supercritical regime. The situation may be dependent also of a possible $H^1$ supercritical regime, as it happens in major dimensions. In the gKdV case, this is not the case, we have always been in a subcritical energy regime. As far as we understand, Theorems \ref{MT2} and \ref{MT2bis} do not require any smallness condition, and no assumptions on the possible rate of decay. The case of odd powers requires more care, since it is usual that in this case new solutions may appear, destroying the convergence analysis \cite{AM,FFMP}. It is believed, from \cite{MuPo1,FFMP}, that strange nondecaying solutions others than solitons may not exist in the full supercritical regime, hence we conjecture that the convergence around spatial zero in Theorems \ref{MT2} and \ref{MT2bis} may hold in the odd power supercritical case.  

\subsection*{Organization of this work} In Section \ref{2} we introduce some preliminaries concerning the linear and nonlinear flow, including the local theory developed by \cite{KPV1} and the minimal rate of blow up. Section \ref{3} deals with the behavior of the solution to \eqref{4gKdV} in far field regions, and the strong far field convergences in Theorem \ref{MT2}. Finally, Section \ref{4} contains the local proofs of normalized decays in Theorem \ref{MT2} in the case of even power nonlinearities.

\subsection*{Acknowledgments} Part of this work was done while C.\ M. was present at BIRS New Synergies in Partial Differential Equations (25w5403) workshop at Banff, whose support is greatly acknowledged. We thank G. Ponce and C. Maul\'en for several useful comments and remarks concerning a first version of this manuscript.

\section{Preliminaries}\label{2}

Recall the classical Gagliardo-Nirenberg inequality
\be\label{GN1}
%\int v^{p+1}(x)dx \leq C(p) \left( \int v^2(x)dx \right)^{\frac{p+3}{4}} \left( \int (\partial_x v)^2(x)dx\right)^{\frac{p-1}{4}}, \quad p>1,
\int |v|^{q}(x)dx \leq C(q) \left( \int v^2(x)dx \right)^{\frac{q+2}{4}} \left( \int (\partial_x v)^2(x)dx\right)^{\frac{q-2}{4}}, \quad q>2,
\ee
valid for $v\in C_0^\infty(\mathbb R)$ and a fortiori for $v\in H^1(\mathbb R)$. We also require the more detailed estimate
\be\label{GN2}
\ba
& \left( \int v^{q}(x)dx \right)^{\frac1q} \leq C \left( \int v^2(x)dx \right)^{\frac{1-\alpha}{2}} \left( \int (D_x^s v)^2(x)dx\right)^{\frac{\alpha}{2}}, \\
& \frac1q = \alpha \left( \frac{1}{2} - s\right) +\frac{1-\alpha}{2} =\frac12 - s\alpha,
\ea
\ee
and finally,
\be\label{GN3}
\sup_{x\in \mathbb R} |v(x)| \leq C  \left( \int v^2(x)dx \right)^{\frac14} \left( \int (\partial_x v)^2(x)dx\right)^{\frac{1}{4}},
\ee  
valid for $v\in H^1(\mathbb R)$. Notice that $\sup$ is well-defined since $v$ is locally H\"older continuous. In this paper, both estimates \eqref{GN1} and \eqref{GN3} will be used but their consequences are equivalent. Note additionally that for larger dimensions the output may be different.

\subsection{Quick review on local well-posedness} Let us review some classical result concerning the local well-posedness of \eqref{4gKdV}. For further details, see \cite{FLP} and \cite{KPV1}. Notice that $s$ below will denote the Sobolev regularity exponent, and will be different from the variable $s=T^*-t$ defined in Theorem \ref{MT2}. 

\begin{lem}[Corollary 2.18 in \cite{KPV1}]\label{LWP}
Let $u_0\in H^s(\mathbb R)$ with $s>s_p$. Then there exists $T=T(\|u_0\|_{H^s})>0$, with $T(\rho;s)\to 0 $ as $\rho\to 0$ and a unique strong solution to \eqref{4gKdV} such that $u\in C([-T,T],H^s(\mathbb R))$. Moreover, the solution satisfies
\be\label{finite_norm}
\ba
& \| D_x^{s_p} \partial_x u\|_{L^\infty_x L^2_T}+\| D_x^{s_p} \partial_x u\|_{L^5_x L^{10}_T}<+\infty,\\
& \| u\|_{L^\infty_T H^s_x} <+\infty,\\
& \| D_x^s \partial_x u\|_{L^\infty_x L^2_T}+\| D_t^{\frac{s}3} \partial_x u\|_{L^\infty_x L^2_T} <+\infty,\\
&  \| D_x^s u\|_{L^5_x L^{10}_T}+\| D_t^{\frac{s}3}  u\|_{L^5_x L^{10}_T} <+\infty,\\
& \left\| D_t^{\frac{s-s_p}3}D_x^{\frac1{10} -\frac2{5(p-1)}} D_t^{\frac3{10} -\frac6{5(p-1)}} u\right\|_{L^{r_p}_x L^{q_p}_T}<+\infty, \\
&  \frac1{r_p} =\frac{2}{5(p-1)} +\frac1{10}, \quad \frac1{q_p} = \frac{3}{10} -\frac4{5(p-1)}.
\ea
\ee
\end{lem}

\begin{proof}[Sketch of proof]
In view of the forthcoming Lemma \ref{BU}, a sketch of proof is necessary. We follow \cite{KPV1}, and assume $u_0\in H^s$, $0\leq s<\frac32$ for simplicity to get Fourier transforms well-defined \cite{KPV1}.  Note that $u_0\in L^2 \cap \dot H^{s_p}$ as well, since $s>s_p$. One has the Duhamel's representation of the problem
\be\label{Duh}
u(t) =\Phi[u](t):= S(t) u_0 -\int_0^t S(t-s) \partial_x(u^p)(s)ds,
\ee
where $S(t)$ denotes the Airy free flow. Let $s\geq 0$ for the moment. Recall the values of $r_p$ and $q_p$ in \eqref{finite_norm}. Define
\be\label{norms}
\ba
& \beta_1^T(w):= \| w\|_{L^\infty_{[-T,T]}L^2_x}=: \| w\|_{L^\infty_T L^2_x},\\
&\beta_2^T(w):= \| \partial_x w\|_{L^\infty_x L^2_T}, \\
&\beta_3^T(w):= \| w\|_{L^5_x L^{10}_T}, \\
&\beta_4^T(w):= \|D_x^{\frac1{10}-\frac2{5(p-1)}} D_t^{\frac{3}{10}-\frac{6}{5(p-1)}}w\|_{L^{r_p}_x L^{q_p}_T}.
%& \Omega^T(w):= \max_{j=1,2,3,4}\{\beta_j^T(w)\}.
\ea
\ee
%\be\label{norms}
%\ba
%& \beta_1^T(w):= \|D_x^{s} w\|_{L^\infty_{[-T,T]}L^2_x}=: \|D_x^{s} w\|_{L^\infty_T L^2_x},\\
%&\beta_2^T(w):= \|D_x^{s} \partial_x w\|_{L^\infty_x L^2_T}, \\
%&\beta_3^T(w):= \|D_x^{s} w\|_{L^5_x L^{10}_T}, \\
%&\beta_4^T(w):= \|D_x^{\frac1{10}-\frac2{5(p-1)}} D_t^{\frac{3}{10}-\frac{6}{5(p-1)}}w\|_{L^{r_p}_x L^{q_p}_T}.
%%& \Omega^T(w):= \max_{j=1,2,3,4}\{\beta_j^T(w)\}.
%\ea
%\ee
%We advance that 
%& X_T:= \{ w \in C([-T,T]; \dot H^{s_p}(\mathbb R))  ~ : ~  \Omega^T(w)<+\infty \}.
%In this space, 
With the norms \eqref{norms}, the following estimates for \eqref{Duh}  are obtained (see Prop. 6.1 in \cite{KPV1}):
\[
\ba
 \beta_1^T(D_x^s\Phi[v]) \leq &~{} \|D_x^{s} S(t) u_0\|_{L^\infty_T L^2_x} + C\| D_x^{s} (v^p)\|_{L^1_xL^2_T} \\
 \leq &~{} \| u_0\|_{\dot H^{s}} + C \| v D_x^{s} (v^{p-1})\|_{L^1_xL^2_T}  + C \| v^{p-1} \|_{L_x^{\frac54}L_T^{\frac52}} \| D_x^{s}v \|_{L^5_xL^{10}_T} , \\
 \leq&~{}  \| u_0\|_{\dot H^{s}} + C \| v\|_{L_x^{\frac{5(p-1)}4}L_T^{\frac{5(p-1)}2}}  \| D_x^{s} (v^{p-1})\|_{L^{\frac{5(p-1)}{5(p-1)-4}}_xL^{\frac{5(p-1)}{5(p-1)-2}}_T}  \\
 &~{}  + C \| v \|_{L_x^{\frac{5(p-1)}4}L_T^{\frac{5(p-1)}2}}^{p-1} \| D_x^{s}v \|_{L^5_xL^{10}_T} , \\
 \leq &~{}  \| u_0\|_{\dot H^{s}} \\
 &~{} + C\left\|D_x^{\frac1{10}-\frac2{5(p-1)}} D_t^{\frac{3}{10}-\frac{6}{5(p-1)}}v\right\|_{L^{r_p}_x L^{q_p}_T}  \| v \|_{L^{\frac{5(p-1)}{4}}_xL^{\frac{5(p-1)}{2}}_T}^{p-2}  \| D_x^{s}v \|_{L^5_xL^{10}_T} \\
&~{} + C\left\|D_x^{\frac1{10}-\frac2{5(p-1)}} D_t^{\frac{3}{10}-\frac{6}{5(p-1)}}v\right\|_{L^{r_p}_x L^{q_p}_T}^{p-1} \| D_x^{s}v \|_{L^5_xL^{10}_T} , \\
 \leq &~{}  \| u_0\|_{\dot H^{s}} + C\left(\beta_4^T(v) \right)^{p-1} \beta_3^T(D_x^s v) .% \| D_x^{s_p} v\|_{L^5_xL^{10}_T},\\
\ea
\]
Similarly, following \cite[Prop 6.1 and Thm. 3.5]{KPV1}
\[
\ba
\beta_2^T(D_x^s \Phi[v]) \leq &~{} \|D_x^{s} \partial_x S(t) u_0\|_{L^\infty_x L^2_T} + C\| D_x^{s} (v^p)\|_{L^1_xL^2_T}\\
 \leq & ~{} C \| u_0\|_{\dot H^{s}} +  C\left(\beta_4^T(v) \right)^{p-1} \beta_3^T(D_x^s v).
\ea
\]
Third, from \cite[Corollary 3.8 and eqn. (6.9)]{KPV1}
\[
\ba
\beta_3^T(D_x^s \Phi[v]) \leq &~{}  \|D_x^{s} S(t) u_0\|_{L^5_x L^{10}_T} + C\| D_x^{s} \partial_x (v^p)\|_{L^{\frac54}_xL^{\frac{10}9}_T}\\
  \leq &~{} C \| u_0\|_{\dot H^{s}}  + C\left(\beta_4^T(v) \right)^{p-1} \beta_2^T(D_x^s v).
\ea
\]
However, notice that from \cite[Corollary 3.16]{KPV1} with $\alpha=\beta=0,$ and following the proof of \cite[eqn. (6.10)]{KPV1},
\[
\ba
\beta_4^T(\Phi[v]) \leq &~{}  C\| u_0\|_{\dot H^{s_p}} + C\| \partial_x(v^p)\|_{L^{r_p'}_xL_T^{q_p'}} \\
 \leq &~{}  C\| u_0\|_{\dot H^{s_p}} + C\left(\beta_4^T(v) \right)^{p-1}\left(\beta_3^T(D_x^{s_p}v) \right)^{1-s_p}\left(\beta_2^T(D_x^{s_p}v) \right)^{s_p}. %, \quad \theta =1- s_p.
%& \qquad \qquad  \leq C \| u_0\|_{H^{s_p}} + C\left(\beta_4^T(v) \right)^{p}  + C\left(\beta_4^T(v) \right)^{p-1} \beta_3^T(v), \\
\ea
\]% + C\left(\beta_4^T(v) \right)^{p}\beta_2^T(v) , \\
Similarly,
\[
\ba
\beta_4^T(D_t^{\frac{s-s_p}3}\Phi[v]) \leq &~{}  C\| u_0\|_{\dot H^{s}} + C\left(\beta_4^T(D_{t}^{\frac{s-s_p}3} v) \right)^{p-1}\left(\beta_3^T(D_x^{s}v) \right)^{1-s_p}\left(\beta_2^T(D_x^{s}v) \right)^{s_p}. %, \quad \theta =1- s_p.
%& \qquad \qquad  \leq C \| u_0\|_{H^{s_p}} + C\left(\beta_4^T(v) \right)^{p}  + C\left(\beta_4^T(v) \right)^{p-1} \beta_3^T(v), \\
\ea
\]
Finally, from \cite[eqn. (3.8) and Thm. A.6]{KPV1}
\[
\ba
 \beta_2^T(D_t^\frac{s}{3}\Phi[v]) \leq &~{}  \|D_t^{\frac{s}3} \partial_x S(t) u_0\|_{L^\infty_x L^2_T} + \left\|D_t^{\frac{s}3} \int_0^t S(t-t')\partial_x^2 (v^p) dt' \right\|_{L^\infty_x L^2_T} \\
 \leq &~{}  C\|u_0\|_{\dot H^s} +C \left\|D_t^{\frac{s}3} (v^p)  \right\|_{L^1_x L^2_T} \\
  \leq &~{}  C\|u_0\|_{\dot H^s} +C \left\| v^{p-1}  \right\|_{L^{\frac54}_x L^{\frac52}_T}  \left\| D_t^{\frac{s}3} v \right\|_{L^5_x L^{10}_T} \\
 \leq &~{}  C\|u_0\|_{\dot H^s} +C   \left(\beta_4^T(v) \right)^{p-1}  \beta_3^T(D_t^{\frac{s}3} v),
\ea
\]
and \cite[eqn. (3.34)]{KPV1}
\[
\ba
 \beta_3^T(D_t^\frac{s}{3}\Phi[v]) \leq  &~{} \|u_0\|_{\dot H^s} +  C\| D_x^s \partial_x (v^{p})\|_{L^{\frac54}_x L^{\frac{10}9}_T} +C \left\| D_t^\frac{s}{3} \partial_x (v^{p}) \right\|_{L^{\frac54}_x L^{\frac{10}9}_T} \\
 \leq  &~{} \|u_0\|_{\dot H^s} + C \left(\beta_4^T(v) \right)^{p-1} \beta_2^T(D_x^s v) . %+\| D_t^\frac{s}{3} \partial_x (v^{p}) \|_{L^{\frac54}_x L^{\frac{10}9}_T}  ,\\
\ea
\]
Finally, using the Sobolev embedding,
\[
\left\|D_x^{\frac1{10}-\frac2{5(p-1)}} D_t^{\frac{3}{10}-\frac{6}{5(p-1)}}w\right\|_{L^{r_p}_x L^{q_p}_T} \leq C T^{\frac{s-s_p}{3}} \left\| D_t^{\frac{s-s_p}{3}} D_x^{\frac1{10}-\frac2{5(p-1)}} D_t^{\frac{3}{10}-\frac{6}{5(p-1)}}w\right\|_{L^{r_p}_x L^{q_p}_T}.
\]
Recall the exponent $(p-1)$ appearing in all the terms $\left(\beta_4^T(v) \right)^{p-1}$ above. This will produce the power $T^{(p-1)\frac{s-s_p}{3}}$ in the fixed point argument. From this point, establishing the local theory is completely similar to the work done in \cite{KPV1,FLP}.
\end{proof}

Notice that from the equality $s_p=\frac12-\frac{2}{p-1}$ it turns out that \eqref{4gKdV} will be always energy subcritical, that is to say, $H^1$ subcritical. Let us consider initial data $u_0\in H^s$, $s>s_p$, with $p\geq 6$. From Lemma \ref{LWP} there exist $T>0$ and a local in time strong solution to \eqref{4gKdV} in $C([-T,T],H^s(\mathbb R))$. Moreover, we can define the maximal lifespan of the solution to be the interval $(-T_*,T^*)\ni 0$, as it is standard in the literature. Having $s$ fixed, assume now that $T^* <+\infty$ is such that 
\be\label{BUs}
\lim_{t\uparrow T^*} \| u(t)\|_{H^s} =+\infty,
\ee
and $\sup_{t\in [0,T] } \| u(t)\|_{H^s} <+\infty$ for all $0<T<T^*$. Naturally from \eqref{mass} one has
\[
\| u(t) \|_{L^2} =\| u_0 \|_{L^2}, \quad t<T^*.
\]
Additionally, for $H^1$ data or smoother, one has the conservation of \eqref{energy} for times $t<T^*$.
%and if $s\geq 1$, $\lim_{t\uparrow T^*} \| u(t)\|_{L^q} =+\infty$, for all $q>2$.

\subsection{Minimal rate of blow up} Now we recall a well-known result concerning the minimal rate of blow up for solutions to \eqref{4gKdV}.

\begin{lem}\label{BU}
Assume $u_0\in H^s$, $s> s_p$ is such that the corresponding local solution $u\in C((-T_*,T^*), H^s(\mathbb R))$ satisfies \eqref{BUs}. Then there exists $C>0$ such that, for all $t<T^*$,
\be\label{BUR}
 \| u(t)\|_{H^s}\geq \frac{C}{(T^*-t)^{\frac{s-s_p}3}}.  
\ee
\end{lem}

Notice that this is the blow up rate also predicted by the self-similar scaling.
%Notice that in the $L^2$ critical case \eqref{BUR} is much well understood, see the works by Martel, Merle and Rapha\"el \cite{MMR1,MMR2,MMR3}, and Martel Pilod \cite{MaPi1,MaPi2}. In the supercritical case, Koch \cite{} and Lang \cite{} provided blow up solutions with blow up.

\begin{proof}[Sketch of proof]
Let $s>s_p$. From a classical argument involving the local well-posedness of the model, in view of \eqref{finite_norm} and the proof of Lemma \ref{LWP}, if $t<T^*$, for all $T<T^*$, one has the following consequence: if for all $M>0$, one has $c\|u(t)\|_{H^s} +c (T-t)^\gamma M^{p} \leq M$, with $\gamma=(p-1)\frac{(s-s_p)}3,$ then $c\|u(t)\|_{H^s} +c (T^*-t)^\gamma M^{p} > M$, for all $M>0$. Choosing $M=2c\|u(t)\|_{H^s}$, it follows that $ (T^*-t)^{\frac{(s-s_p)}3}\| u(t)\|_{H^s}\geq C$, leading to \eqref{BUR}.
\end{proof}

\section{Decay in far field regions}\label{3}

In this section we prove \eqref{FF1}-\eqref{FF2} and \eqref{FF3}-\eqref{FF4}. First of all, we prove \eqref{FF1} and \eqref{FF3}. 

\begin{lem}
There exists $C_0>0$ such that  
\be\label{right1}
\lim_{t\uparrow T^*} \|u(t)\|_{L^2(x\geq  \beta_1(t) )} =0, \quad \hbox{with} \quad \beta_1(t) := C_0 \left( t + \int_0^t \| \partial_x u(s)\|_{L^2}^{\frac{p-1}2} ds \right).
\ee
\end{lem}

\begin{remark}
The choice of $\beta_1(t)$ is in some sense optimal. Indeed, in the case where $T^*=+\infty$ and $\sup_{t\in\mathbb R_+}\| \partial_x u(s)\|_{L^2}^{\frac{p-1}2} <+\infty$, one recovers the classical bound $\beta_1(t) \gtrsim t$ depending on the size of the solution.
\end{remark}

\begin{proof} This proof is inspired in Martel and Merle's work \cite{MMnon}. First of all, assume $T^* =+\infty$. Define $\beta_1(t)$ as in \eqref{right1}, with $C_0>0$ to be determined below. For times $ t,t_0 \geq 2>0$, and $L>0$ large enough, set
\be\label{def_varphi0}
\varphi_{t_0}(t,x):= \frac12\left( 1+ \tanh\left(\frac{x -\beta_1(t_0) - \frac12(\beta_1(t)- \beta_1(t_0))}{L}  \right) \right).
\ee
Notice that for $t,x$ fixed, $\lim_{t_0\to +\infty}\varphi_{t_0}(t,x) =0$. Using $\varphi_{t_0}$ as before, consider
\[
\mathcal J_{t_0}(t):=\frac12 \int u^2(t,x) \varphi_{t_0}(t,x)dx.
\]
We have
\[
\frac{d}{dt}\mathcal J_{t_0}(t) =   \frac12 \int \partial_t \varphi_{t_0} u^2 -\frac32\int \partial_x \varphi_{t_0} (\partial_x u)^2 +\frac12\int \partial_x^3 \varphi_{t_0} u^2 +\frac{p}{p+1} \int\partial_x \varphi_{t_0} u^{p+1}.
\]
We first notice that $\partial_x \varphi_{t_0}>0$ and
\[
\frac12 \int \partial_t \varphi_{t_0} u^2 =-\frac12 \beta_1'(t)\int \partial_x \varphi_{t_0} u^2.
\]
Second, 
\[
\left| \frac12\int \partial_x^3 \varphi_{t_0} u^2 \right| \leq \frac{2}{L^2}  \int \partial_x \varphi_{t_0} u^2.
\]
Notice that from \eqref{GN3},%since $\|u(t)\|_{L^\infty}$ is bounded in time with bound depending on $\|u_0\|_{H^1}$,
\[
\ba
\left| \frac{p}{p+1} \int\partial_x \varphi_{t_0} u^{p+1} \right| \leq &~{} \frac{p}{p+1} \|u(t)\|_{L^\infty}^{p-1} \int\partial_x \varphi_{t_0} u^2\\
 \leq &~{} C \| \partial_x u(t)\|_{L^2}^{\frac{p-1}2} \int\partial_x \varphi_{t_0} u^2.
\ea
\]
We conclude that 
\[
\frac{d}{dt}\mathcal J_{t_0}(t) \leq  -\left( \frac12 \beta_1'(t) -C \| \partial_x u(t)\|_{L^2}^{\frac{p-1}2} - \frac{2}{L^2} \right)\int \partial_x \varphi_{t_0} u^2 -\frac32\int \partial_x \varphi_{t_0} (\partial_x u)^2 .
\]
By the choice of $\beta_1(t)$, for some $C_0>0$ sufficiently large, we obtain
\[
\frac{d}{dt}\mathcal J_{t_0}(t) \leq  0.
\]
Consequently, for $t_0\geq 2$, $\mathcal J_{t_0}(t_0) \leq \mathcal J_{t_0}(2)$. We have
\be\label{caso_base}
\ba
& \int u^2(t_0,x) \left( 1+ \tanh\left(\frac{x -\beta_1(t_0)}{L}  \right) \right) dx\\
&\quad  \leq \int u^2(2,x)\left( 1+ \tanh\left(\frac{x - \frac12 (\beta_1(t_0) +\beta(2))}{L}  \right) \right) dx.
\ea
\ee
Sending $t_0\to +\infty,$ we conclude \eqref{right1}  in the case $T^*=+\infty$.

The proof of  \eqref{right1} in the case $T^*<+\infty$ is completely similar, with some minor differences. First of all, let $0<t_1\leq t_0,t <T^*$. Define $\varphi_{t_0}$ as in \eqref{def_varphi0}. Now follow the lines of the proof in the previous case, up to \eqref{caso_base}, to obtain
\be\label{cota_externa}
\ba
& \int u^2(t_0,x) \left( 1+ \tanh\left(\frac{x -\beta_1(t_0)}{L}  \right) \right) dx\\
&\quad  \leq \int u^2(t_1,x)\left( 1+ \tanh\left(\frac{x - \frac12 (\beta_1(t_0) +\beta_1(t_1))}{L}  \right) \right) dx.
\ea
\ee
 If $\lim_{t_0\uparrow T^*} \beta_1(t_0) =+\infty$, then the proof is very similar to the previous case. If now $\lim_{t_0\uparrow T^*} \beta_1(t_0) <+\infty$, given $\varepsilon>0$, choosing $C_0(\varepsilon, \| u(t_1)\|_{L^2})>0$ large enough, we will obtain 
\[
\ba
& \limsup_{t\uparrow T^*}\int u^2(t_0,x) \left( 1+ \tanh\left(\frac{x -\beta_1(t_0)}{L}  \right) \right) dx\\
&\quad  \leq\limsup_{t\uparrow T^*} \int u^2(t_1,x)\left( 1+ \tanh\left(\frac{x - \frac12 (\beta_1(t_0) +\beta_1(t_1))}{L}  \right) \right) dx<\varepsilon.
\ea
\]
\end{proof}

\subsection{Decay in left far field region}

The purpose of this subsection is to show strong $L^2$ decay in the far left region, namely, the proof of \eqref{FF2} and \eqref{FF4}. In recent works \cite{MPS,MMPP1,MMPP2}, a similar result was proved but only in finite portions of this region. Here we improve that result by considering the whole left far region at once. This is done by taking a suitable modification of the weight function taken in \cite{MPS}.

\begin{lem}\label{lejos_izquierda}
Assume that $T^*=+\infty$. For any $C_1>0$ and $\eta>0$, one has
\be\label{left1}
\lim_{t\to +\infty} \|u(t)\|_{L^2(x\leq - \beta_2(t))} =0, \quad \beta_2(t):= C_1  \left(1 + \| \partial_x u(t)\|_{L^2}^{\frac{p-1}2} \right) t \log^{1+\eta}t.
\ee
\end{lem}

\begin{remark}
The choice of $\beta_2(t)$ is in some sense optimal. Indeed, in the case where $T^*=+\infty$ and $\sup_{t\in\mathbb R_+}\| \partial_x u(s)\|_{L^2}^{\frac{p-1}2} <+\infty$, one recovers the classical bound $\beta_2(t) \sim t  \log^{1+\eta}t$ already found in previous works by the authors. % depending on the size of the solution.
\end{remark}

%\begin{remark}
%From the proof it will become clear that the same proof works for the KdV equation (quadratic nonlinearity) with minor modifications in some estimates and constants. Consequently, \eqref{left} is in some sense sharp, because there is a modification of the KdV model (the so-called Gardner model) which is integrable and has breather solutions moving to the left with speed $\sim -1$. These breathers are localized in the region $x\sim -t$  and do not decay. This problem does not appear when the data is chosen small, since in that case Garder breathers are not present. See \cite{Alejo,MuPo1} and references therein for mode details on these solutions. 
%\end{remark}

\begin{proof}[Proof of Lemma \ref{lejos_izquierda}]
Let $C_1>0$ and $\eta>0$ be fixed. Consider any $\varepsilon>0$ and $t_1>t_0>2$ sufficiently large such that  
\be\label{small_epsilon}
C(C_1) \int_{t_0}^{+\infty} \frac{dt}{t \log^{1+\eta} t} < \varepsilon.
\ee
Here $C(C_1)$ is a fixed constant that will be explicit below. We will also need a suitable cut-off function $\chi$, satisfying
\be\label{chi}
\begin{cases}
\chi\in C^\infty(\R), \quad 0\leq \chi \leq 1 \quad \hbox{ in}\quad \R, \\
\chi(s) \equiv 1 \quad \hbox{if}\quad s\leq -1, \quad  \chi(s) \equiv 0 \quad \hbox{if}\quad s\geq 0, \\
\chi' (s) <0, \quad \hbox{in}\quad (-1,0),\\
|\chi^{(k)} (s)|\leq 2^k, \quad \hbox{in}\quad (-1,0), \quad k=1,2,3.
\end{cases}
\ee
Let $t\geq 3.$ Let $\mu(t)$ be a smooth increasing function to be defined later. For $\chi$ as in \eqref{chi}, let
\[
\varphi(t,x):= \chi\left( \frac{x+\frac12(\mu(t)+ \mu(t_1))}{\mu(t)}\right).
\]
For further consideration, we have 
\be\label{derivative_t_chi}
\begin{aligned}
\partial_t \varphi(t,x) = &~{}  -\frac{\mu'(t)}{\mu(t)} \left(\frac{x+\frac12(\mu(t)+ \mu(t_1))}{\mu(t)} \right)\chi'\left( \frac{x+\frac12(\mu(t)+ \mu(t_1))}{\mu(t)}\right) \\
&~{} +\frac12 \frac{\mu'(t)}{\mu(t)}\chi'\left( \frac{x+\frac12(\mu(t)+ \mu(t_1))}{\mu(t)}\right),
\end{aligned}
\ee
and for $k=1,2,3$,
\be\label{derivative_x_chi}
\partial_x^{k} \varphi(t,x) = \frac1{\mu^{k}(t)} \chi^{(k)}\left( \frac{x+\frac12(\mu(t)+ \mu(t_1))}{\mu(t)}\right).
\ee
On the other hand, for any $0<t<+\infty$ we have
\[
\frac{d}{dt} \mathcal J(t)  =   \frac12 \int \partial_t \varphi u^2 -\frac32\int \partial_x \varphi (\partial_x u)^2 +\frac12\int \partial_x^3 \varphi u^2 +\frac{p}{p+1} \int\partial_x \varphi u^{p+1}.
\]
Now we estimate the small terms in the previous identity. First, using \eqref{derivative_x_chi} and \eqref{chi}, % and the $L^\infty_t H^1_x$ bound \eqref{K0} on the solution  (here $n.m$ means $n\times m$),
\[
\left| \frac32\int \partial_x \varphi (\partial_x u)^2 \right| \lesssim \frac{1}{\mu(t)} \int (\partial_x u)^2 \lesssim \frac{\| \partial_x u(t)\|_{L^2}^2}{\mu(t)}.
\] 
Similarly,
\[
\left| \frac12\int \partial_x^3 \varphi u^2 \right| \lesssim \frac{1}{\mu^3(t)} \int u^2 \lesssim \frac{1}{\mu^3(t)},
\] 
and from \eqref{GN3},
\[
\ba
\left| \frac{p}{p+1}\int \partial_x \varphi u^{p+1} \right| \leq &~{} \frac{2p}{(p+1)\mu(t)} \|u(t)\|_{L^\infty}^{p-1} \int u^2 \\
\lesssim &~{} \frac{1}{\mu(t)} \| \partial_x u(t)\|_{L^2}^{\frac{p-1}2}.
\ea
\] 
Now we have two cases:
\begin{itemize}
\item Case $\sup_{t\geq 0} \| \partial_x u(t)\|_{L^2} <+\infty$. 
\item Case $\limsup_{t\to +\infty} \| \partial_x u(t)\|_{L^2} =+\infty$.
\end{itemize}
In the first case, we simply choose 
\be\label{mu_acotado}
\mu(t) := \frac12C_1 t\log^{1+\eta} t.
\ee
Since $\chi'$ is nonzero on $[-1,0]$, we have from \eqref{derivative_t_chi} $\partial_t \varphi(t,x)  \leq 0$, 
\[
\frac{d}{dt} \mathcal J(t)  \lesssim \frac{1}{t\log^{1+\eta} t}.
\]
Integrating in $[t_0,t_1]$, and using \eqref{small_epsilon}, we obtain
\[
\mathcal J(t_1) \leq \mathcal J(t_0) + \frac12C(C_1) \int_{t_0}^{t_1} \frac{1}{t\log^{1+\eta} t} < \mathcal J(t_0) + \frac 12\varepsilon.
\]
We first conclude that 
\[
\int u^2(t_1,x)\chi\left( \frac{x+\mu(t_1)}{\mu(t_1)}\right)dx \leq  \int u^2(t_0,x)\chi\left( \frac{x+\frac12(\mu(t_0) + \mu(t_1))}{\mu(t_0)}\right) dx +\varepsilon,
\]
and noticing that in the region $\frac{x+\mu(t_1)}{\mu(t_1)} \leq -1$ one has $\chi\equiv 1$,
\[
\int_{x<-2\mu(t_1)} u^2(t_1,x) dx \leq \int u^2(t_0,x)\chi\left( \frac{x+\frac12(\mu(t_0) + \mu(t_1))}{\mu(t_0)}\right) dx +\varepsilon.
\]
Sending $t_1$ to infinity, we get
\[
\ba
& \limsup_{t_1\to +\infty}\int_{x<-2\mu(t_1)} u^2(t_1,x) dx \\
&\quad \leq \lim_{t_1\to +\infty}\int u^2(t_0,x)\chi\left( \frac{x+\frac12(\mu(t_0) + \mu(t_1))}{\mu(t_0)}\right) dx +\varepsilon =\varepsilon.
\ea
\]
Since $\varepsilon>0$ is arbitrary, we conclude that 
\[
\lim_{t\to +\infty}\int_{x<- C_1 t \log^{1+\eta}t} u^2(t,x) dx =0.
\]
This proves \eqref{left1} in the first case after choosing $\beta_2(t):= 2\mu(t).$ 

In the second case, let $\mu(t)$ be smooth positive, unbounded and increasing such that 
\[
\mu(t) \geq \frac12C_1  t \log^{1+\eta}t \left(1 + \| \partial_x u(t)\|_{L^2}^{\frac{p-1}2} \right).
\]
 Then one can see that the previous argument in the case $\| \partial_x u(t)\|_{L^2}^{\frac{p-1}2}$ holds, obtaining 
\[
\frac{d}{dt} \mathcal J(t)  \lesssim \frac{1}{t\log^{1+\eta} t}.
\]
The rest of the proof is similar. This ends the proof of \eqref{left1} in the remaining case, after choosing $\beta_2(t):= 2\mu(t).$
\end{proof}

Now we consider the case where $T^*<+\infty$.

\begin{lem}\label{lejos_izquierda 2}
Assume that $T^*<+\infty$ and that $u$ in Theorem \ref{MT2} satisfies the lower bound blow up rate
\be\label{cond_rara}
\|\partial_x u(t)\|_{L^2} \geq \frac{c_0}{(T^*-t)^{\frac4{3(p-1)}}},%\|\partial_x u(t)\|_{L^2} \geq \frac{1}{(T^*-t)^{\frac4{3(p-1)}} |\log^{\frac{8}{3(p-1)}}(T^*-t)|}
\ee
for some $c_0>0.$ Then, for any $C_1>0$, $\eta>0$, and  $\beta_2(t)$ smooth such that $\beta_2(t) \geq C_1\|\partial_x u(t)\|_{L^2}^{\frac{p-1}{2}} (T^*-t) |\log^{1+\eta}(T^*-t)|$, one has
\be\label{left22}
\ba
& \lim_{t\uparrow T^*} \|u(t)\|_{L^2(x\leq -   \beta_2(t) )} =0.
\ea
\ee
\end{lem}

Note that in the case $T^*=+\infty$ and $\sup_{t\geq 0} \| \partial_x u(t)\|_{L^2} <+\infty$ we recover the choice of $\mu(t)$ \eqref{mu_acotado}, after the suitable change in the time variable.

\begin{remark}\label{rem_clave}
From \eqref{BUR}, in the case $s=1$ we get 
\[
\|\partial_x u(t)\|_{L^2} \geq \frac{c_0}{(T^*-t)^{\frac13(1-s_p)}} =\frac{c_0}{(T^*-t)^{\frac{p+3}{6(p-1)}}} \gtrsim \frac{1}{(T^*-t)^{\frac4{3(p-1)}}}
\]
 for all $p \geq 5$. Therefore, \eqref{cond_rara} is satisfied in the critical case (with a suitable constant), and in the supercritical case $p>5$ for any possible constant $c_0$ if $t$ is sufficiently close to $T^*$.
\end{remark}

\begin{proof}[Proof of Lemma \ref{lejos_izquierda 2}] 
We follow the proof of Lemma \ref{lejos_izquierda} with some key modifications. First of all,  let $C_1>0$ and $\eta>0$ be fixed as previously mentioned. With no loss of regularity, we can assume $T^*>1$, otherwise we redefine some logarithms below. Consider any $\varepsilon>0$ and $1<t_0<T^*$ sufficiently close to $T^*$ such that  
\be\label{small_epsilon_2}
C(C_1) \int_{t_0}^{T^*} \frac{dt}{(T^*-t) \log^{1+\eta} (T^*-t)} < \varepsilon.
\ee
Again, $C(C_1)$ is a fixed constant to be found below. Consider the same cut-off function as in \eqref{chi}. Let $0<t<T^*$ and $\mu(t)$ be a smooth increasing function to be defined later. For $\chi$ as in \eqref{chi}, let
\[
\varphi(t,x):= \chi\left( \frac{x+\frac12(\mu(t)+ \mu(t_1))}{\mu(t)}\right).
\]
Notice that \eqref{derivative_t_chi} and \eqref{derivative_x_chi} holds for $t<T^*$. Also, since $\chi'$ is nonzero on $[-1,0]$, and $\mu'(t)>0$, we have from \eqref{derivative_t_chi} $\partial_t \varphi(t,x)  \leq 0$.
%For further consideration, we have 
%\be\label{derivative_t_chi}
%\begin{aligned}
%\partial_t \varphi(t,x) = &~{}  -\frac{\mu'(t)}{\mu(t)} \left(\frac{x+\frac12(\mu(t)+ \mu(t_1))}{\mu(t)} \right)\chi'\left( \frac{x+\frac12(\mu(t)+ \mu(t_1))}{\mu(t)}\right) \\
%&~{} +\frac12 \frac{\mu'(t)}{\mu(t)}\chi'\left( \frac{x+\frac12(\mu(t)+ \mu(t_1))}{\mu(t)}\right),
%\end{aligned}
%\ee
%and for $k=1,2,3$,
%\be\label{derivative_x_chi}
%\partial_x^{k} \varphi(t,x) = \frac1{\mu^{k}(t)} \chi^{(k)}\left( \frac{x+\frac12(\mu(t)+ \mu(t_1))}{\mu(t)}\right).
%\ee
On the other hand, for any $0<t<T^*$ we have the classical Kato identity
\[
\frac{d}{dt} \mathcal J(t)  =   \frac12 \int \partial_t \varphi u^2 -\frac32\int \partial_x \varphi (\partial_x u)^2 +\frac12\int \partial_x^3 \varphi u^2 +\frac{p}{p+1} \int\partial_x \varphi u^{p+1}.
\]
Exactly as in the previous lemma, %Now we estimate the small terms in the previous identity. First, using \eqref{derivative_x_chi} and \eqref{chi}, % and the $L^\infty_t H^1_x$ bound \eqref{K0} on the solution  (here $n.m$ means $n\times m$),
\[
\left| \frac32\int \partial_x \varphi (\partial_x u)^2 \right| \lesssim \frac{1}{\mu(t)} \int (\partial_x u)^2 \lesssim \frac{\| \partial_x u(t)\|_{L^2}^2}{\mu(t)},
\] 
\[
\left| \frac12\int \partial_x^3 \varphi u^2 \right| \lesssim \frac{1}{\mu^3(t)} \int u^2 \lesssim \frac{1}{\mu^3(t)},
\] 
and from \eqref{GN3},
\[
\ba
\left| \frac{p}{p+1}\int \partial_x \varphi u^{p+1} \right| \leq &~{} \frac{2p}{(p+1)\mu(t)} \|u(t)\|_{L^\infty}^{p-1} \int u^2 \\
\lesssim &~{} \frac{1}{\mu(t)} \| \partial_x u(t)\|_{L^2}^{\frac{p-1}2}.
\ea
\] 
Recall that $\eta>0$ is a small number. We choose $\mu(t)$ smooth increasing such that 
\be\label{def_mu}
\mu(t) \geq \frac12C_1\|\partial_x u(t)\|_{L^2}^{\frac{p-1}{2}} (T^*-t) |\log^{1+\eta}(T^*-t)|.% (T^*-t)^{\frac13}\log^{1+\eta}(T^*-t). %, \frac{(T^*-t) }{\|\partial_x u(t)\|_{L^2}^{\frac{p-1}{2}}}\right\}.
\ee
This choice reflects a competition between the term $\|\partial_x u(t)\|_{L^2}^{\frac{p-1}{2}}$ that diverges to infinity as $t\uparrow T^*$ (at least for a sequence of times), and the term $(T^*-t) |\log^{1+\eta}(T^*-t)|$ that converges to zero as $t\uparrow T^*$. We do not assume any final choice on the long time behavior of $\|\partial_x u(t)\|_{L^2}^{\frac{p-1}{2}} (T^*-t) |\log^{1+\eta}(T^*-t)|$. Therefore, the choice of $\mu(t)$ must be taken with care. First of all,
\[
\frac{1}{\mu(t)} \| \partial_x u(t)\|_{L^2}^{\frac{p-1}2} \lesssim  \frac1{(T^*-t) |\log^{1+\eta}(T^*-t)|}\in L^1([0,T^*)).
\]
Additionally, under \eqref{cond_rara} one has $\| \partial_x u(t)\|_{L^2}^{\frac{p-1}2} \geq \| \partial_x u(t)\|_{L^2}^{2}$ as $t\uparrow T^*$, $p\geq 5$, so that $\frac{1}{\mu(t)} \| \partial_x u(t)\|_{L^2}^{2} \in L^1([0,T^*))$. Finally, from \eqref{def_mu} and \eqref{cond_rara},
\[
\ba
\mu^3(t) \sim &~{}  \|\partial_x u(t)\|_{L^2}^{\frac{3(p-1)}{2}} (T^*-t)^{3} |\log^{3(1+\eta)}(T^*-t)| \\
\geq &~{} (T^*-t) |\log^{1+\eta}(T^*-t)|,
\ea
\]
and then $\frac{1}{\mu^3(t)} \in  L^1([0,T^*))$. We conclude 
\[
\frac{d}{dt} \mathcal J(t)  \lesssim \frac{1}{(T^*-t)\log^{1+\eta} (T^*-t)}.
\]
Integrating in $[t_0,t_1]$, and using \eqref{small_epsilon_2}, we obtain for $C(C_0)$ coming from above
\[
\mathcal J(t_1) \leq \mathcal J(t_0) + \frac12 C(C_0) \int_{t_0}^{t_1}\frac{1}{(T^*-t)\log^{1+\eta} (T^*-t)} < \mathcal J(t_0) + \frac 12\varepsilon.
\]
We first conclude that 
\[
\int u^2(t_1,x)\chi\left( \frac{x+\mu(t_1)}{\mu(t_1)}\right)dx \leq  \int u^2(t_0,x)\chi\left( \frac{x+\frac12(\mu(t_0) + \mu(t_1))}{\mu(t_0)}\right) dx +\varepsilon,
\]
and noticing that in the region $\frac{x+\mu(t_1)}{\mu(t_1)} \leq -1$ one has $\chi\equiv 1$,
\[
\int_{x<-2\mu(t_1)} u^2(t_1,x) dx \leq \int u^2(t_0,x)\chi\left( \frac{x+\frac12(\mu(t_0) + \mu(t_1))}{\mu(t_0)}\right) dx +\varepsilon.
\]
Now we have two cases, where $\mu(t)$ is unbounded, and where $\mu(t)$ remains bounded, In the first case, sending $t_1$ to infinity, we get
\[
\ba
& \limsup_{t_1\to +\infty}\int_{x<-2\mu(t_1)} u^2(t_1,x) dx \\
&\quad \leq \lim_{t_1\to +\infty}\int u^2(t_0,x)\chi\left( \frac{x+\frac12(\mu(t_0) + \mu(t_1))}{\mu(t_0)}\right) dx +\varepsilon =\varepsilon.
\ea
\]
Since $\varepsilon>0$ is arbitrary, we conclude that 
\[
\lim_{t\to +\infty}\int_{x<- \mu(t)} u^2(t,x) dx =0.
\]
This ends the proof of \eqref{left22} after choosing $\beta_2(t):=\mu(t)$. In the second case, we proceed as in \eqref{cota_externa} by making $C(C_0)$ large enough we can get a similar decay to zero.
\end{proof}

%\begin{remark}
%From the proof it will become clear that the same proof works for the KdV equation (quadratic nonlinearity) with minor modifications in some estimates and constants. Consequently, \eqref{left} is in some sense sharp, because there is a modification of the KdV model (the so-called Gardner model) which is integrable and has breather solutions moving to the left with speed $\sim -1$. These breathers are localized in the region $x\sim -t$  and do not decay. This problem does not appear when the data is chosen small, since in that case Garder breathers are not present. See \cite{Alejo,MuPo1} and references therein for mode details on these solutions. 
%\end{remark}

\section{Behavior on compact regions}\label{4}

In this section we prove \eqref{MT21}, \eqref{MT01bis} and \eqref{MT22}. Recall that $p=2n$, $n\geq 3$. The proof is divided in two cases.

\subsection{Case $T^*$ finite} Assume $t<T^*$. Let $s:= T^*-t>0$. Then the limit $t\uparrow T^*$ is equivalent to $s\downarrow 0$. With some abuse of notation, we denote $u(s):=u(T^*-s)$, taking care of the moment where one takes time derivative.  Let%We shall follow \cite{MMPP1} with key differences.
\be\label{varphi3}
\varphi(s,x):= \frac1{\theta(s)} \tanh\left( \frac{x-\mu(s)}{\lambda_1(s)}\right)\sech^2\left( \frac{x-\mu(s)}{\lambda_2(s)}\right),
\ee
where $\beta(s) \geq \| \partial_x u(s) \|_{L^2}^{n-1}$ is a smooth, decreasing function such that $\lim_{s\downarrow 0}\beta(s) =+\infty$, and
\be\label{conds}
\ba
\theta(s)= &~{} (s\beta)^{\frac13}(s)\log^2 s \\
\lambda_2(s)= &~{} \frac{\theta^2 (s)}{\log^4 s}  =(s\beta)^{\frac23}(s),\\
\lambda_1(s)= &~{} \frac{(s\beta)^{\frac23}(s)}{\log^2s}  \ll \lambda_2(s),%(s\beta)^{\frac23}(s) =\frac{s^{\frac23}}{\log s\beta^{\frac13}(s)} , \quad \frac{C}{(s\beta)^{\frac13}(s)\log^2 s \lambda_1^{\frac52}(s)} \gg \lambda_2(s), \\
\ea
\ee
and
\be\label{varphi3bis}
 \mu(s) \hbox{ such that }  |\mu'(s)|\lesssim \frac{\beta^{\frac23}}{s^{\frac13} \log s}.
\ee
From now on, we assume without loss of generality that $0<s<S^*:=\min\{\frac12 ,T^*\}$. We will measure integrability in the interval $(0,S^*]$. First of all, notice that %we have
\be\label{derm1}
\left| \int \varphi u\right| \lesssim  \frac{\lambda_2^{\frac12}(s)}{\theta(s)} \| u \|_{L^\infty(0,T^*; L^2(\mathbb R))} \lesssim \frac{\lambda_2^{\frac12}(s)}{\theta(s)} \lesssim \frac1{\log^2s}.
\ee
Notice that $\varphi$ in \eqref{varphi3} is smooth, localized and bounded in space, for each fixed $0<s<S^*$. Now we have by classical integration by parts for $0<s<S^*$, namely far from the blow up time,
\be\label{der0}
-\frac{d}{ds} \int \varphi (s)u = \frac{d}{dt} \int \varphi (T^*-t) u = -\int \partial_s \varphi u +\int \partial_x^3 \varphi u +\int \partial_x\varphi u^{2n}.
\ee
Now we recall that
\be\label{der1}
\ba
\partial_s \varphi = &~{} -\frac{\theta'(s)}{\theta^2(s)} \tanh\left( \frac{x-\mu(s)}{\lambda_1(s)}\right)\sech^2\left( \frac{x-\mu(s)}{\lambda_2(s)}\right) \\
&~{} - \frac{\lambda_1'(s)}{\theta(s)\lambda_1(s)} \left( \frac{x-\mu(s)}{\lambda_1(s)}\right) \sech^2\left( \frac{x-\mu(s)}{\lambda_1(s)}\right)\sech^2\left( \frac{x-\mu(s)}{\lambda_2(s)}\right)\\
&~{} -\frac{\mu'(s)}{\theta(s)\lambda_1(s)} \sech^2\left( \frac{x-\mu(s)}{\lambda_1(s)}\right)\sech^2\left( \frac{x-\mu(s)}{\lambda_2(s)}\right)\\
&~{} - \frac{\lambda_2'(s)}{\theta(s)\lambda_2(s)} \tanh\left( \frac{x-\mu(s)}{\lambda_1(s)}\right)\left( \frac{x-\mu(s)}{\lambda_2(s)}\right) (\sech^2)'\left( \frac{x-\mu(s)}{\lambda_2(s)}\right)\\
&~{} -\frac{\mu'(s)}{\theta(s) \lambda_2(s)} \tanh\left( \frac{x-\mu(s)}{\lambda_1(s)}\right)(\sech^2)'\left( \frac{x-\mu(s)}{\lambda_2(s)}\right).
\ea
\ee
Additionally, 
\be\label{der2}
\ba
\partial_x \varphi = &~{} \frac1{\theta(s)\lambda_1(s)} \sech^2 \left( \frac{x-\mu(s)}{\lambda_1(s)}\right)\sech^2\left( \frac{x-\mu(s)}{\lambda_2(s)}\right) \\
&~{} +\frac1{\theta(s)\lambda_2(s)} \tanh\left( \frac{x-\mu(s)}{\lambda_1(s)}\right)(\sech^2)'\left( \frac{x-\mu(s)}{\lambda_2(s)}\right),
\ea
\ee
and
\be\label{der3}
\ba
\partial_x^3 \varphi = &~{} \frac1{\theta(s)\lambda_1^3(s)} (\sech^2)''\left( \frac{x-\mu(s)}{\lambda_1(s)}\right)\sech^2\left( \frac{x-\mu(s)}{\lambda_2(s)}\right) \\
&~{} +\frac3{\theta(s)\lambda_1^2(s)\lambda_2(s)} (\sech^2)'\left( \frac{x-\mu(s)}{\lambda_1(s)}\right)(\sech^2)'\left( \frac{x-\mu(s)}{\lambda_2(s)}\right)\\
&~{} +\frac3{\theta(s)\lambda_1(s)\lambda_2^2(s)} \sech^2\left( \frac{x-\mu(s)}{\lambda_1(s)}\right)(\sech^2)''\left( \frac{x-\mu(s)}{\lambda_2(s)}\right)\\
&~{} +\frac1{\theta(s)\lambda_2^3(s)} \tanh\left( \frac{x-\mu(s)}{\lambda_1(s)}\right)(\sech^2)'''\left( \frac{x-\mu(s)}{\lambda_2(s)}\right).
\ea
\ee
Now we compute each term in \eqref{der0}. Using \eqref{der1} and Cauchy-Schwarz, 
\[
\ba
\left| \int \partial_s \varphi u \right| \lesssim &~{} \frac{\lambda_2^{\frac12}(s)}{s\theta(s)} \| u \|_{L^\infty(0,T^*; L^2(\mathbb R))} +\frac{\lambda_1^{\frac12}(s)}{s\theta(s)} \| u \|_{L^\infty(0,T^*; L^2(\mathbb R))} \\
&~{} + \frac{|\mu'(s)|}{\theta(s)\lambda_1^{\frac12}(s)} \| u \|_{L^\infty(0,T^*; L^2(\mathbb R))} +\frac{\lambda_2^{\frac12}(s)}{s\theta(s)} \| u \|_{L^\infty(0,T^*; L^2(\mathbb R))}  \\
&~{} + \frac{|\mu'(s)|}{\theta(s)\lambda_2^{\frac12}(s)} \| u \|_{L^\infty(0,T^*; L^2(\mathbb R))}. 
\ea
\]
From the $L^2$ boundedness of the solution,
\be\label{cota11}
\ba
 \left| \int \partial_s \varphi u \right| \lesssim &~{} \frac1{\theta(s)} \left( \frac{\lambda_2^{\frac12}(s)}{s}  + \frac{|\mu'(s)|}{\lambda_1^{\frac12}(s)}  \right). 
\ea
\ee
Using \eqref{der3}, and the fact that $\lambda_1(s)\ll \lambda_2(s)$ in \eqref{conds},
\be\label{cota12}
\ba
\left| \int  \partial_x^3 \varphi u \right| \lesssim &~{} \frac1{\theta(s)}\left(  \frac1{\lambda_1^{\frac52}(s)}+  \frac1{\lambda_1^{\frac32}(s) \lambda_2(s)} + \frac1{\lambda_1^{\frac12}\lambda_2^{2}(s)} + \frac1{\lambda_2^{\frac52}(s)} \right) \| u \|_{L^\infty(0,T^*; L^2(\mathbb R))}.
%\lesssim &~{}{\color{red}  \frac1{\theta(s)\lambda_1^{\frac52}(s)}.}
\ea
\ee
Here we can check the validity of the choices in \eqref{conds} with a simple example: the quantity $\frac1{\theta(t)\lambda_2^{\frac52}(t)} $ satisfies
\[
\frac1{\theta(t)\lambda_2^{\frac52}(t)}  = \frac1{(s\beta)^2(s) \log^2 s} \leq \frac1{s\log^2 s}, 
\]
provided $\beta(s) \geq \frac1{\sqrt{s}}$. From Remark \ref{rem_clave}, we get $\beta(s) \geq \| \partial_x u(s)\|_{L^2}^{n-1} \gtrsim s^{-\frac{(n-1)(2n+3)}{6(2n-1)}}$. Therefore, the condition $s^{-\frac{(n-1)(2n+3)}{6(2n-1)}} \gg s^{-\frac12}$ is valid if $n> \frac52$, precisely above the $L^2$ critical case $p=2n=5$. A similar situation occurs with the remaining terms. We consider the most dangerous term, coming from the fact that $\lambda_1(s)\ll \lambda_2(s)$. We have
\[
\frac1{\theta(t)\lambda_1^{\frac52}(t)}  = \frac{\log^{\frac52}s}{(s\beta)^2(s)} \ll \frac1{s\log^2 s}, 
\]
since $n\geq3$, see Lemma \ref{BU}. 

From \eqref{der2},  and the fact that $p=2n$,
\be\label{cota13}
\ba
& \int \frac1{\theta(s)\lambda_1(s)} \sech^2 \left( \frac{x-\mu(s)}{\lambda_1(s)}\right)\sech^2\left( \frac{x-\mu(s)}{\lambda_2(s)}\right) u^{2n} \\
& \qquad  \geq \frac1{\theta(s)\lambda_1(s)}   \int_{|x-\mu(s)| \leq \lambda_1(s)}  u^{2n}(t,x)dx.
\ea
\ee
Finally, again using \eqref{der2}, and the fact that $p=2n$,
\[
\ba
 \left| \int   \frac1{\theta(s) \lambda_2(s)} \tanh\left( \frac{x-\mu(s)}{\lambda_1(s)}\right)(\sech^2)'\left( \frac{x-\mu(s)}{\lambda_2(s)}\right) u^{2n} \right|  \lesssim \frac1{\theta(s) \lambda_2(s)} \int u^{2n}. 
\ea
\]
Now we use \eqref{GN1} and $\| u\|_{L^\infty(0,T^*; L^2(\mathbb R))}$ bounded to obtain
\be\label{cota14}
\ba
 &\left| \int   \frac1{\theta(s) \lambda_2(s)} \tanh\left( \frac{x-\mu(s)}{\lambda_1(s)}\right)(\sech^2)'\left( \frac{x-\mu(s)}{\lambda_2(s)}\right) u^{2n} \right| \lesssim  \frac{\| \partial_x u(s) \|_{L^2}^{n-1}}{\theta(s) \lambda_2(s)}. 
\ea
\ee
Gathering \eqref{cota11}, \eqref{cota12}, \eqref{cota13} and \eqref{cota14}, we obtain for some $C,c_0>0$,
\[
\ba
-\frac{d}{ds} \int \varphi u \geq  &~{} \frac{c_0}{\theta(s)\lambda_1(s)}   \int_{|x-\mu(s)| \leq \lambda_1(s)}  u^{2n}(t,x)dx \\
&~{} -\frac{C}{\theta(s)} \left( \frac{\lambda_2^{\frac12}(s)}{s}  + \frac{|\mu'(s)|}{\lambda_1^{\frac12}(s)}  \right) - \frac{C}{\theta(s)\lambda_1^{\frac52}(s)} -\frac{C\beta(s)}{\theta(s) \lambda_2(s)}.
\ea
\]
Now we use \eqref{conds} to get
\[
\frac{1}{\theta(s)\lambda_1(s)}  =\frac{1}{s\beta(s)}, 
\]
\[
\frac{\lambda_2^{\frac12}(s)}{\theta(s)s} = \frac1{s\log^2s},
\]
\[
\frac{|\mu'(s)|}{\theta(s) \lambda_1^{\frac12}(s)} \lesssim \frac{\beta^{\frac23} \log s}{ s^{\frac13} \log s (s\beta)^{\frac13}(s)\log^2 s (s\beta)^{\frac13}(s)} = \frac1{s\log^2s},
\]
%and since $s\beta(s)$ cannot go below $s^{1-\frac{p+3}{6(p-1)}} =s^{\frac{5p-9}{6(p-1)}}$
%\[
%\frac{1}{\theta(s)\lambda_1^{\frac52}(s)} =\frac1{(s\beta(s))^2 \log^2s}\ll \frac1{s\log^2s}, \qquad \hbox{ (using Lemma \ref{BU})},
%\]
and
\[
\frac{\beta(s)}{\theta(s) \lambda_2(s)} = \frac1{s\log^2s}.
\]
Therefore, 
\[
\ba
-\frac{d}{ds} \int \varphi u \geq  &~{} \frac{c_0}{s\beta(s)}   \int_{|x-\mu(s)| \leq \lambda_1(s)}  u^{2n}(t,x)dx  -\frac{C}{s\log^2 s} .
\ea
\]
Integrating in time and using \eqref{derm1}, we conclude \eqref{MT21} with $\lambda_3:=\lambda_1$, $\beta_3=\beta$ and $\mu_3:=\mu$.

\subsection{Case $T^*$ infinite}   Here we have two subcases:
\begin{itemize}
\item $\sup_{t\geq 0} \| \partial_x u(t) \|_{L^2} <+\infty$;
\item $\limsup_{t\to+\infty} \| \partial_x u(t) \|_{L^2} = +\infty$.
\end{itemize}
The first case is very similar to the one proved in \cite[Theorem 1.9]{MMPP1} and we left the proof to the reader. This proves \eqref{MT01bis}. The second case recognizes the fact that infinite time blow up may occur. We assume no increasing behavior of the $L^2$ norm of the gradient, only $\limsup_{t\to+\infty} \| \partial_x u(t) \|_{L^2} = +\infty$.  Now we concentrate on the second case. From now on, we assume $t\gg 1$.  As in the previous subsection, define 
\be\label{varphi3a}
\varphi(t,x):= \frac1{\theta(t)} \tanh\left( \frac{x-\mu(t)}{\lambda_1(t)}\right)\sech^2\left( \frac{x-\mu(t)}{\lambda_2(t)}\right),
\ee
where $\beta(t) \geq \| \partial_x u(t) \|_{L^2}^{n-1}$ is a smooth, increasing function such that $\lim_{t\to +\infty}\beta(t) =+\infty$, and
\be\label{condsa}
\ba
&\theta(t)=(t\beta)^{\frac13}(t)\log^2 t \\
&\lambda_2(t)= \frac{\theta^2 (t)}{\log^4 t}  = (t\beta)^{\frac23}(t),\\
&\lambda_1(t)= \frac{(t\beta)^{\frac23}(t)}{\log t}   \ll \lambda_2(t),\\%(s\beta)^{\frac23}(t) =\frac{s^{\frac23}}{\log s\beta^{\frac13}(t)} , \quad \frac{C}{(s\beta)^{\frac13}(t)\log^2 s \lambda_1^{\frac52}(t)} \gg \lambda_2(t), \\
& \mu(t) \hbox{ be such that } |\mu'(t)| \lesssim \frac{(t\beta)^{\frac23}(t)}{t}.
\ea
\ee
Notice that $\varphi$ in \eqref{varphi3} is smooth, localized and bounded in space, for each time $t$ large.  First of all, notice that from \eqref{varphi3a}, \eqref{condsa} and the conservation of mass we have
\be\label{derm1a}
\left| \int \varphi u\right| \lesssim  \frac{\lambda_2^{\frac12}(t)}{\theta(t)} \| u \|_{L^\infty(0,+\infty; L^2(\mathbb R))} \lesssim \frac{\lambda_2^{\frac12}(t)}{\theta(t)} \lesssim \frac1{\log^2t}.
\ee
Therefore, $\int \varphi u$ is well-defined and bounded in time. Now we have
\be\label{der0a}
 \frac{d}{dt} \int \varphi u = \int \partial_t \varphi u +\int \partial_x^3 \varphi u +\int \partial_x\varphi u^{2n}.
\ee
As in the previous case, 
\be\label{der1a}
\ba
\partial_t \varphi = &~{} -\frac{\theta'(t)}{\theta^2(t)} \tanh\left( \frac{x-\mu(t)}{\lambda_1(t)}\right)\sech^2\left( \frac{x-\mu(t)}{\lambda_2(t)}\right) \\
&~{} - \frac{\lambda_1'(t)}{\theta(t)\lambda_1(t)} \left( \frac{x-\mu(t)}{\lambda_1(t)}\right) \sech^2\left( \frac{x-\mu(t)}{\lambda_1(t)}\right)\sech^2\left( \frac{x-\mu(t)}{\lambda_2(t)}\right)\\
&~{} -\frac{\mu'(t)}{\theta(t)\lambda_1(t)} \sech^2\left( \frac{x-\mu(t)}{\lambda_1(t)}\right)\sech^2\left( \frac{x-\mu(t)}{\lambda_2(t)}\right)\\
&~{} - \frac{\lambda_2'(t)}{\theta(t)\lambda_2(t)} \tanh\left( \frac{x-\mu(t)}{\lambda_1(t)}\right)\left( \frac{x-\mu(t)}{\lambda_2(t)}\right) (\sech^2)'\left( \frac{x-\mu(t)}{\lambda_2(t)}\right)\\
&~{} -\frac{\mu'(t)}{\theta(t) \lambda_2(t)} \tanh\left( \frac{x-\mu(t)}{\lambda_1(t)}\right)(\sech^2)'\left( \frac{x-\mu(t)}{\lambda_2(t)}\right).
\ea
\ee
Additionally, 
\be\label{der2a}
\ba
\partial_x \varphi = &~{} \frac1{\theta(t)\lambda_1(t)} \sech^2 \left( \frac{x-\mu(t)}{\lambda_1(t)}\right)\sech^2\left( \frac{x-\mu(t)}{\lambda_2(t)}\right) \\
&~{} +\frac1{\theta(t)\lambda_2(t)} \tanh\left( \frac{x-\mu(t)}{\lambda_1(t)}\right)(\sech^2)'\left( \frac{x-\mu(t)}{\lambda_2(t)}\right),
\ea
\ee
and
\be\label{der3a}
\ba
\partial_x^3 \varphi = &~{} \frac1{\theta(t)\lambda_1^3(t)} (\sech^2)''\left( \frac{x-\mu(t)}{\lambda_1(t)}\right)\sech^2\left( \frac{x-\mu(t)}{\lambda_2(t)}\right) \\
&~{} +\frac3{\theta(t)\lambda_1^2(t)\lambda_2(t)} (\sech^2)'\left( \frac{x-\mu(t)}{\lambda_1(t)}\right)(\sech^2)'\left( \frac{x-\mu(t)}{\lambda_2(t)}\right)\\
&~{} +\frac3{\theta(t)\lambda_1(t)\lambda_2^2(t)} \sech^2\left( \frac{x-\mu(t)}{\lambda_1(t)}\right)(\sech^2)''\left( \frac{x-\mu(t)}{\lambda_2(t)}\right)\\
&~{} +\frac1{\theta(t)\lambda_2^3(t)} \tanh\left( \frac{x-\mu(t)}{\lambda_1(t)}\right)(\sech^2)'''\left( \frac{x-\mu(t)}{\lambda_2(t)}\right).
\ea
\ee
Now we compute each term in \eqref{der0a}. Using \eqref{der1a} and Cauchy-Schwarz,
\[
\ba
\left| \int \partial_t \varphi u \right| \lesssim &~{} \frac{\lambda_2^{\frac12}(t)}{t \theta(t)} \| u \|_{L^\infty(0,+\infty; L^2(\mathbb R))} +\frac{\lambda_1^{\frac12}(t)}{t \theta(t)} \| u \|_{L^\infty(0,+\infty; L^2(\mathbb R))} \\
&~{} + \frac{|\mu'(t)|}{\theta(t)\lambda_1^{\frac12}(t)} \| u \|_{L^\infty(0,+\infty; L^2(\mathbb R))} +\frac{\lambda_2^{\frac12}(t)}{t \theta(t)} \| u \|_{L^\infty(0,+\infty; L^2(\mathbb R))}  \\
&~{} + \frac{|\mu'(t)|}{\theta(t)\lambda_2^{\frac12}(t)} \| u \|_{L^\infty(0,+\infty; L^2(\mathbb R))}. 
\ea
\]
Therefore, from the boundedness of the $L^2$ norm and the fact that $\lambda_2(t)\gg \lambda_1(t)$,
\be\label{cota11a}
\ba
\left| \int \partial_t \varphi u \right| \lesssim &~{} \frac1{\theta(t)} \left( \frac{\lambda_2^{\frac12}(t)}{t}  + \frac{|\mu'(t)|}{\lambda_1^{\frac12}(t)}  \right). 
\ea
\ee
Using \eqref{der3a} and the previous arguments,
\be\label{cota12a}
\ba
\left| \int  \partial_x^3 \varphi u \right| \lesssim &~{} \frac1{\theta(t)}\left(  \frac1{\lambda_1^{\frac52}(t)}+  \frac1{\lambda_1^{\frac32}(t) \lambda_2(t)} + \frac1{\lambda_1^{\frac12}\lambda_2^{2}(t)} + \frac1{\lambda_2^{\frac52}(t)} \right) \| u \|_{L^\infty(0,+\infty; L^2(\mathbb R))}\\
\lesssim &~{} \frac1{\theta(t)\lambda_1^{\frac52}(t)}.
\ea
\ee
From \eqref{der2a},  and the fact that $p=2n$,
\be\label{cota13a}
\ba
& \int \frac1{\theta(t)\lambda_1(t)} \sech^2 \left( \frac{x-\mu(t)}{\lambda_1(t)}\right)\sech^2\left( \frac{x-\mu(t)}{\lambda_2(t)}\right) u^{2n} \\
& \qquad  \geq \frac1{\theta(t)\lambda_1(t)}   \int_{|x-\mu(t)| \leq \lambda_1(t)}  u^{2n}(t,x)dx.
\ea
\ee
Finally, again using \eqref{der2a}, and the fact that $p=2n$,
\[
\ba
 \left| \int   \frac1{\theta(t) \lambda_2(t)} \tanh\left( \frac{x-\mu(t)}{\lambda_1(t)}\right)(\sech^2)'\left( \frac{x-\mu(t)}{\lambda_2(t)}\right) u^{2n} \right|  \lesssim \frac1{\theta(t) \lambda_2(t)} \int u^{2n}. 
\ea
\]
Now we use \eqref{GN1} with $q=2n$ and $\| u\|_{L^\infty(0,+\infty; L^2(\mathbb R))}$ bounded to obtain
\be\label{cota14a}
\ba
 &\left| \int   \frac1{\theta(t) \lambda_2(t)} \tanh\left( \frac{x-\mu(t)}{\lambda_1(t)}\right)(\sech^2)'\left( \frac{x-\mu(t)}{\lambda_2(t)}\right) u^{2n} \right|  \lesssim \frac{\| \partial_x u(t) \|_{L^2}^{n-1}}{\theta(t) \lambda_2(t)}. 
\ea
\ee
Gathering \eqref{cota11a}, \eqref{cota12a}, \eqref{cota13a} and \eqref{cota14a}, we obtain for some $C,c_0>0$,
\[
\ba
\frac{d}{dt} \int \varphi u \geq  &~{} \frac{c_0}{\theta(t)\lambda_1(t)}   \int_{|x-\mu(t)| \leq \lambda_1(t)}  u^{2n}(t,x)dx \\
&~{} -\frac{C}{\theta(t)} \left( \frac{\lambda_2^{\frac12}(t)}{t}  + \frac{|\mu'(t)|}{\lambda_1^{\frac12}(t)}  \right) - \frac{C}{\theta(t)\lambda_1^{\frac52}(t)} -\frac{C\beta(t)}{\theta(t) \lambda_2(t)}.
\ea
\]
Now we use \eqref{condsa} to get
\[
\frac{1}{\theta(t)\lambda_1^{\frac52}(t)} = \frac{1}{\frac{(t\beta)^{\frac53}(t)}{\log^{\frac52} t}(t\beta)^{\frac13}(t)\log^2 t} = \frac{\log^{\frac12}t}{t^2\beta^2(t)} \ll \frac1{t\log^2t},
\]
\[
\frac{|\mu'(t)|}{\theta(t)\lambda_1^{\frac12}(t)} \lesssim \frac{\beta(t) \log^{\frac12} t}{(t\beta)^{\frac13}(t) (t\beta)^{\frac13}(t) (t\beta)^{\frac13}(t) \log^2 t} =  \frac{1}{t \log^{\frac32}t}, % \ll \frac1{t\log^{\frac32}t},
\]
and finally
\[
\frac{\beta(t)}{\theta(t) \lambda_2(t)} =\frac1{t\log^2t}.
\]
We conclude that
\[
\ba
\frac{d}{dt} \int \varphi u \geq  &~{} \frac{c_0}{t\beta (t) \log t}   \int_{|x-\mu(t)| \leq \lambda_1(t)}  u^{2n}(t,x)dx  -\frac{C}{t\log^{\frac32} t} .
\ea
\]
Integrating in time and using \eqref{derm1a}, we conclude \eqref{MT21} in the case $T^* = +\infty$ with $\beta_3:=\beta$, $\lambda_3:=\lambda_1$ and $\mu_3:= \mu$.

\subsection{Final remark} We conclude this paper by mentioning that the approach followed in this paper can be translated to other models where one has a nontrivial positive nonlinearity. The existence of at least one uniformly bounded norm is crucial, in this case, the conservation of the $L^2$ norm is key to conclude the arguments. We expect to investigate consequences of this fact in a forthcoming work.

\end{document}